\documentclass[12pt]{article}
\usepackage{a4}
\usepackage{amsthm}
\usepackage{amsfonts}
\usepackage{amssymb}
\usepackage{amsmath}
\usepackage{cite}
\usepackage{epsfig}
\usepackage{mathtools}
\usepackage{hyperref}

\newtheorem{theorem}{Theorem}
\newtheorem{corollary}[theorem]{Corollary}
\newtheorem{lemma}[theorem]{Lemma}
\newtheorem{proposition}[theorem]{Proposition}
\newtheorem{definition}[theorem]{Definition}
\newtheorem{problem}{Problem}
\newcommand\DD{{\cal D}}
\newcommand\dd{\,\mbox{d}}
\newcommand\NN{{\mathbb N}}
\newcommand\ZZ{{\mathbb Z}}
\newcommand\RR{{\mathbb R}}

\DeclareTextCompositeCommand{\v}{OT1}{l}{l\nobreak\hspace{-.1em}'}
\DeclareTextCompositeCommand{\v}{OT1}{t}{t\nobreak\hspace{-.1em}'\nobreak\hspace{-.15em}}

\begin{document}
\title{Towards characterizing locally common graphs\thanks{The work of the second author has received funding from the European Research Council (ERC) under the European Union's Horizon 2020 research and innovation programme (grant agreement No 648509). This publication reflects only its authors' view; the European Research Council Executive Agency is not responsible for any use that may be made of the information it contains. The first and the second authors were also supported by the MUNI Award in Science and Humanities of the Grant Agency of Masaryk University. The third author was supported in part by the Slovenian Research Agency (research program P1-0383 and research project N1-0160).}}

\author{Robert Hancock\thanks{Institut f\"ur Informatik, University of Heidelberg, Im Neuenheimer Feld 205, 69120, Heidelberg, Germany. Previous affiliation: Faculty of Informatics, Masaryk University, Botanick\'a 68A, 602 00 Brno, Czech Republic. E-mail: {\tt hancock@informatik.uni-heidelberg.de}}\and
        Daniel Kr{\'a}\v{l}\thanks{Faculty of Informatics, Masaryk University, Botanick\'a 68A, 602 00 Brno, Czech Republic, and Mathematics Institute, DIMAP and Department of Computer Science, University of Warwick, Coventry CV4 7AL, UK. E-mail: {\tt dkral@fi.muni.cz}.}\and
        Matja\v z Krnc\thanks{Faculty of Mathematics, Natural Sciences and Information Technologies, University of Primorska, Glagolja\v ska 8, SI-6000 Koper, Slovenia. E-mail:  {\tt matjaz.krnc@upr.si}.}\and
        Jan Volec\thanks{Department of Mathematics, Faculty of Nuclear Sciences and Physical Engineering, Czech Technical University in Prague, Trojanova 13, 120 00 Prague, Czech Republic. E-mail: {\tt jan@ucw.cz}.}}
\date{} 
\maketitle

\begin{abstract}
A graph $H$ is \emph{common}
if the number of monochromatic copies of $H$ in a 2-edge-coloring of the complete graph
is asymptotically minimized by the random coloring.
The classification of common graphs is one of the most intriguing problems in extremal graph theory.
We study the notion of weakly locally common graphs considered by Cs\'oka, Hubai and Lov\'asz [arXiv:1912.02926],
where the graph is required to be the minimizer with respect to perturbations of the random $2$-edge-coloring.
We give a complete analysis of the $12$ initial terms in the Taylor series
determining the number of monochromatic copies of $H$ in such perturbations and
classify graphs $H$ based on this analysis into three categories:
\begin{itemize}
\item graphs of Class I are weakly locally common,
\item graphs of Class II are not weakly locally common, and
\item graphs of Class III cannot be determined to be weakly locally common or not based on the initial $12$ terms.
\end{itemize}
As a corollary, we obtain new necessary conditions on a graph to be common and
new sufficient conditions on a graph to be not common.
\end{abstract}

\section{Introduction}\label{sec:intro}

Ramsey's theorem states that, for any graph $H$,
every 2-edge-coloring of the complete graph $K_n$ contains a monochromatic copy of $H$, provided that $n$ is sufficiently large.
The natural quantitative question stemming from this classical theorem is the following:
\emph{What is the minimum number of monochromatic copies of $H$ contained in a $2$-edge-coloring of $K_n$?
In particular, is the minimum achieved by the random $2$-edge-coloring of $K_n$?}
Our main result
is a complete analysis of the initial $12$ terms of the polynomial determining the number of monochromatic copies of $H$
in a perturbation of the random $2$-edge-coloring of $K_n$.

We next put our results in a broader context.
A graph $H$ is \emph{common} 
if the number of monochromatic copies of $H$
is asymptotically minimized by the random $2$-edge-coloring of $K_n$.
The notion of common graphs originated in the 1980s but can be traced to even older results.
Indeed, the classical result of Goodman~\cite{Goo59} implies that the graph $K_3$ is common,
which led Erd\H{o}s~\cite{Erd62} to conjecture that every complete graph is common;
this conjecture was extended by Burr and Rosta~\cite{BurR80} to all graphs.
Sidorenko~\cite{Sid89} disproved the Burr--Rosta Conjecture by showing that a triangle with a pendant edge is not common, and
around the same time, Thomason~\cite{Tho89} disproved the original conjecture of Erd\H{o}s
by establishing that $K_p$ is not common for any $p\geq4$.
Several additional constructions showing that $K_p$ is not common for $p\geq4$ have since been found~\cite{Tho97,FraR93,Fra02}, and
more generally, Jagger, \v{S}\v{t}ov\'{\i}\v{c}ek and Thomason~\cite{JagST96} showed that
no graph containing a copy of $K_4$ is common.
Determining the asymptotics of the minimum number of monochromatic copies of $K_4$
remains an open problem despite many partial results~\cite{Gir79,Nie,Spe11}.

A characterization of common graphs is one of the most intriguing problems in extremal graph theory;
there is not even a conjecture for a possible characterization of common graphs.
On one hand,
common graphs include odd cycles~\cite{Sid89} and even wheels~\cite{JagST96}, and
additional examples of common graphs can be obtained by certain gluing operations~\cite{JagST96,Sid96}.
Only recently, an example of a common graph with chromatic number larger than $3$ was identified:
the $5$-wheel was shown to be common in~\cite{HatHKNR12} using Razborov's flag algebra method introduced in~\cite{Raz07}.
On the negative side,
Fox~\cite{Fox07} proved that every (connected) non-bipartite graph is a subgraph of a connected graph that is not common.
We also refer the reader to~\cite{CumY11,KraNNVW+} for results on the analogous concept involving more colors.

Common graphs are very closely linked to Sidorenko graphs.
A graph $H$ is \emph{Sidorenko} if the number of copies of $H$ in any graph $G$
is asymptotically bounded from below by the number of copies of $H$ in the random graph of the same density as $G$.
It easily follows that every Sidorenko graph is bipartite and
a convexity argument yields that every Sidorenko graph is common.
A well known conjecture of Sidorenko~\cite{Sid91,Sid93},
which is equivalent to an earlier conjecture of Erd\H{o}s and Simonovits~\cite{ErdS84},
asserts that in fact every bipartite graph is Sidorenko.
So, if true, then every bipartite graph would be common.
Many families of bipartite graphs are known to be Sidorenko~\cite{LiS11,ConFS10,Hat10,Sze15,ConKLL18,ConL17,ConL18},
however, the complete solution of the conjecture seems to be out of reach.

Sidorenko's Conjecture is well-understood in the local setting,
i.e., when perturbations of the random graph with a given edge density are considered.
Lov\'asz~\cite{Lov11} showed that no fixed perturbation decreases the number of copies of a graph $H$
if and only if $H$ is a tree or its girth is even.
In the language of theory of graph limits, which we introduce in Section~\ref{sec:prelim},
this result asserts that for every such graph $H$ and every kernel $U$ with $t(K_2,U)=0$,
there exists $\varepsilon_0>0$ such that
\[t(H,1/2)\le t(H,1/2+\varepsilon U) \quad \mbox{for every }\varepsilon\in (0,\varepsilon_0).\]
Fox and Wei~\cite{FoxW17} strengthened the result of Lov\'asz and proved the following:
for every such graph $H$, there exists $\varepsilon_0>0$ such that
$t(H,1/2)\le t(H,1/2+U)$ for every $U$ with $t(K_2,U)=0$, $\|U\|_{\square}\le\varepsilon_0$ and $\|U\|_{\infty}\le 1/2$.
In other words, any large graph close to a random graph has at least the same density of $H$ as a random graph.

\subsection{Locally common graphs}

We study the local version of the notion of common graphs,
which has recently been introduced by Cs\'oka, Hubai and Lov\'asz~\cite{CsoHL19}.
As in the case of Sidorenko's Conjecture, several notions of locally common graphs can be considered.
The one that we study here is the following notion,
which is referred to as weakly locally common in~\cite{CsoHL19} and
which we simply refer to as to locally common for brevity throughout the paper:
a graph $H$ is \emph{locally common} if for every kernel $U$,
there exists $\varepsilon_0>0$ such that
\[ 2t(H,1/2)\le t(H,1/2+\varepsilon U)+t(H,1/2-\varepsilon U) \quad \mbox{for every }\varepsilon\in (0,\varepsilon_0).\]
Cs\'oka et al.~\cite{CsoHL19} showed that every graph containing $K_4$ is locally common in this sense.
In the sense analogous to that considered by Fox and Wei~\cite{FoxW17},
the result of Franek and R\"odl~\cite{FraR93} yields that $K_4$ is not locally common and
Cs\'oka et al.~\cite{CsoHL19} established that in fact any graph containing $K_4$ is not locally common in this stronger sense.

We provide a strong partial characterization of locally common graphs,
which also suggests that the characterization of common graphs is likely to be very complex.
To be more precise, for every graph $H$ and kernel $U$,
we analyze the function $t(H,1/2+\varepsilon U)+t(H,1/2-\varepsilon U)$,
which is a polynomial of $\varepsilon$, and
give a complete characterization of its possible coefficients up to the term of $\varepsilon^{12}$ (inclusively).
This characterization is presented in Theorems~\ref{thm:deck8}, \ref{thm:deck10} and \ref{thm:deck12}.
In particular, we split graphs $H$ into three classes:
\begin{itemize}
\item graphs of Class I are locally common,
\item graphs of Class II are not locally common, and
\item graphs of Class III admit a kernel $U$ such that
the coefficients of the initial twelve terms in $t(H,1/2+\varepsilon U)+t(H,1/2-\varepsilon U)$ are zero and
there is no kernel $U$ that witnesses that $H$ is not locally common based on one of the initial twelve terms.
\end{itemize}
In other words,
if $H$ is of Class III, then it is not possible to decide whether $H$ is locally common or not solely by analyzing the initial twelve terms in the expression $t(H,1/2+\varepsilon U)+t(H,1/2-\varepsilon U)$.
To establish the classification,
in Section~\ref{sec:kernel},
we develop techniques for constructing kernels $U$ with strong control of the change of the number odd cycles passing through given vertices.
We believe that these techniques will be useful to the further study of locally common graphs and common graphs in general.

While the actual characterization given in Theorems~\ref{thm:deck8}, \ref{thm:deck10} and \ref{thm:deck12} is complex,
which is caused by the involved nature of the problem, we state some of the corollaries here.
Let $C_k\oplus C_{\ell}$ be the graph obtained by identifying one vertex of $C_k$ and one vertex of $C_{\ell}$.
The following sufficient conditions on a graph $H$ to be locally common are implied by our characterization:
\begin{itemize}
\item $H$ contains $C_4$ or $C_6$.
\item $H$ contains $C_8$ and $C_3\oplus C_3$.
\item $H$ contains $C_8$ and two edge-disjoint $C_3$'s but it does not contain $C_3\oplus C_5$.
\end{itemize}
We remark that the first condition was already established by Cs\'oka et al.~\cite[Theorem 4.1]{CsoHL19} who proved the following:
if $H$ is a graph with even girth $g$ that
does not contain two cycles of different odd lengths $\ell_1$ and $\ell_2$ sharing at most one vertex such that $\ell_1+\ell_2\le g$ and
also does not contain two cycles of the same odd length $\ell$ sharing at most one vertex such that $2\ell<g$,
then $H$ is locally common.
On the negative side, Cs\'oka et al.~\cite{CsoHL19} showed that the graph $C_3\oplus C_5$ is not locally common.
More generally, we show that the following are sufficient conditions on a graph $H$ to be not locally common:
\begin{itemize}
\item $H$ contains $C_3\oplus C_5$ but does not contain $C_4$, $C_6$ or $C_3\oplus C_3$.
\item $H$ contains vertex disjoint $C_3$ and $C_5$ but it does not contain $C_4$, $C_6$ or two edge-disjoint $C_3$'s.
\end{itemize}

The examples above may suggest that whether the graph $H$ is locally common or not
is determined by the presence or the absence of particular subgraphs.
While this is indeed the case for subgraphs with at most eight edges,
the situation already becomes more involved when $10$-edge subgraphs are considered.
For example, suppose that a graph $H$ does not contain $C_4$, $C_6$, $C_8$, $C_3\oplus C_3$, $C_3\oplus C_5$, $C_3\oplus C_7$, or $C_3\oplus P_2\oplus C_3$ (the last graph is depicted in Figure~\ref{fig:princ8}), $H$ does contain two edge-disjoint $C_3$'s and $C_{10}$, and
let $s_{33}$, $s_{35}$ and $s_{55}$ be the numbers of subgraphs of $H$ isomorphic to the graphs $C_3\oplus P_4\oplus C_3$, $C_3\oplus P_2\oplus C_5$ and $C_5\oplus C_5$ (see Figure~\ref{fig:princ10}), respectively.
Theorem~\ref{thm:deck10} yields that $H$ is locally common if and only if $4s_{33}s_{55}\ge \left(s_{35}\right)^2$.

This paper is structured as follows. 
In Section~\ref{sec:prelim}, we introduce the notation and basic terminology from the theory of graph limits, and
in Section~\ref{sec:sums},
we prove auxiliary number theory results required to develop our tools presented in Section~\ref{sec:kernel}.
In Sections~\ref{sec:deck8}, \ref{sec:deck10} and \ref{sec:deck12}
we provide classifications of locally common graphs with respect to subgraphs with $8$, $10$ and $12$ edges, respectively, and
we describe which graphs can be concluded to be locally common (Class I), which to be not locally common (Class II), and
which belong to neither of the two classes (Class III).
We finish with presenting two open questions concerning locally common graphs suggested by our work in Section~\ref{sec:conclusion}.
 
 \section{Preliminaries}
\label{sec:prelim}

In this section, we fix notation used throughout the paper.
We start with some basic notation and introduce more specialized notation in subsections.
The set of the first $n$ positive integers is denoted by $[n]$.
All graphs considered here are finite and simple.
If $G$ is a graph, then $V(G)$ and $E(G)$ is the vertex set and the edge set of $G$.
The \emph{order} of $G$, i.e., its number of vertices, is denoted by $|G|$, and
its \emph{size}, i.e., its number of edges, by $\|G\|$.
The complete graph of order $n$ is denoted by $K_n$, the $n$-vertex cycle by $C_n$ and the $n$-edge path by $P_n$.
If $G$ and $H$ are two graphs, then $G\cup H$ is the graph obtained as a disjoint union of $G$ and $H$.
If $G$ and $H$ are two vertex transitive graphs,
then $G\oplus H$ is the graph obtained from $G\cup H$ by identifying one vertex of $G$ with one vertex of $H$, and
$G\oplus P_n\oplus H$ is the graph obtained from $G\cup P_n\cup H$ by identifying one vertex of $G$ with one end-vertex of the path $P_n$ and
one vertex of $H$ with the other end-vertex of $P_n$.
We will also use the notation $G\oplus H$ when one or both $G$ and $H$ are not vertex transitive
if the vertex of $G$ and the vertex of $H$ to be identified are clear from the context.
A \emph{homomorphism} from a graph $H$ to a graph $G$
is a function $f:V(H)\to V(G)$ such that $f(u)f(v)\in E(G)$ for every edge $uv\in E(H)$, and
the \emph{homomorphism density} of $H$ in $G$, which is denoted by $t(H,G)$,
is the probability that a random function from $V(H)$ to $V(G)$ is a homomorphism,
i.e., it is the number of homomorphisms from $H$ to $G$ divided by $|G|^{|H|}$.

\subsection{Decks}
\label{subsec:prelim1}

In this subsection, we introduce notation related to decks of graphs,
which play a crucial role in determining whether a graph is locally common or not.
An \emph{$\ell$-deck} is any multiset of $\ell$-edge graphs, and
the \emph{$\ell$-deck} of a graph $G$, which is denoted by $G[\ell]$,
is the multiset of all $\ell$-edge subgraphs of $G$.
If $\DD$ is an $\ell$-deck and $H$ is an $\ell$-edge graph,
we write $s_{\DD}(H)$ for the number of copies of $H$ that $\DD$ contains.
More generally, we can define $s_{\DD}(H)$ for a graph $H$ with less than $\ell$ edges as
the number of $\|H\|$-edge subgraphs of the graphs in $\DD$ that are isomorphic to $H$.
We next define an $\ell'$-deck $\DD'$ of an $\ell$-deck $\DD$ for $\ell'\le\ell$:
it is simply a union of all $\ell'$-decks of graphs contained in $\DD$ (with their multiplicities).
Note that $s_{\DD'}(H)=s_{\DD}(H)$ for every $\ell'$-edge graph $H$.
Observe that
if $\DD'$ is the $\ell'$-deck of the $\ell$-deck $G[\ell]$ of a graph $G$ on $m$ edges, then
\[s_{\DD'}(H)=\binom{m-\ell'}{\ell-\ell'}s_{G[\ell']}(H)\]
for every $\ell'$-edge graph $H$,
i.e., the multiplicities of graphs in $G[\ell']$ and the $\ell'$-deck of $G[\ell]$
differ by the multiplicative constant independent of $H$.
We will later recall that the $\ell$-th coefficient in the polynomial $t(G,1/2+\varepsilon U)$
is a linear combination of $s_{G[\ell]}(H)$ for $\ell$-edge graphs $H$,
i.e., its sign is the same regardless of
whether we consider it directly with the $\ell$-deck of $G$ or with the $\ell$-deck of another deck of $G$.

\subsection{Graphons and kernels}
\label{subsec:prelim2}

In this part of Section~\ref{sec:prelim}, we introduce basic terminology from the theory of graph limits.
A \emph{graphon} is a measurable function $W:[0,1]^2\to [0,1]$ that is symmetric,
i.e., $W(x,y)=W(y,x)$ for all $(x,y)\in [0,1]^2$.
Intuitively (and quite imprecisely),
a graphon can be thought of as a continuous variant of the adjacency matrix of a graph.
The graphon that is equal to $p\in [0,1]$ everywhere is called the $p$-constant graphon;
when there will be no confusion, we will just use $p$ to denote such a graphon.
The notion of \emph{homomorphism density} extends to graphons by setting
\begin{equation}
t(H,W) := \int_{[0,1]^{V(H)}} \prod_{uv\in E(H)} W(x_u,x_v) \dd x_{V(H)}\label{eq:tHW}
\end{equation}
for a graph $H$ and graphon $W$, and
we define the \emph{density} of a graphon $W$ to be $t(K_2,W)$.

The quantity $t(H,W)$ has a natural interpretation in terms of sampling a random graph according to $W$:
a $W$-random graph of order $n\in\NN$, which is denoted by $G_{n,W}$, is obtained
by sampling $n$ independent uniform random points $x_1,\ldots,x_n$ from the interval $[0,1]$ and
joining the $i$-th and $j$-th vertices of $G$ by an edge with probability $W(x_i,x_j)$.
It can be shown that the following holds for every graph $H$ with probability one:
\[\lim_{n\to\infty}t(H,G_{n,W})=t(H,W).\]
A sequence $(G_i)_{i\in\NN}$ of graphs is \emph{convergent}
if the sequence $(t(H,G_i))_{i\in\NN}$ converges for every graph $H$.
A simple diagonalization argument implies that every sequence of graphs has a convergent subsequence.
We say that a graphon $W$ is a \emph{limit} of a convergent sequence $(G_i)_{i\in\NN}$ of graphs if
\[\lim_{i\to\infty}t(H,G_i)=t(H,W)\]
for every graph $H$.
One of the crucial results in graph limits, due to Lov\'asz and Szegedy~\cite{LovS06}, is that
every convergent sequence of graphs has a limit.
Hence, a graph $H$ is \emph{Sidorenko} if and only if $t(H,W)\ge t(K_2,W)^{\|H\|}$ for every graphon $W$, and
$H$ is \emph{common} if and only if $t(H,W)+t(H,1-W)\ge 2^{1-\|H\|}$ for every graphon $W$.

A perturbation of a graphon can be described by a kernel.
Formally, a \emph{kernel} is a bounded measurable symmetric function $U:[0,1]^2\to\RR$, and
we define the homomorphism density of $H$ in $U$ as in \eqref{eq:tHW}, i.e.,
\begin{equation}
t(H,U) := \int_{[0,1]^{V(H)}} \prod_{uv\in E(H)} U(x_u,x_v) \dd x_{V(H)}.\label{eq:tHU}
\end{equation}
If $W$ is a graphon and $U$ is a kernel, then it holds that
\begin{equation}
t(H,p+\varepsilon U)=p^{\|H\|}+\sum_{k\in[\|H\|]}p^{\|H\|-k}\varepsilon^k\sum_{H'\in H[k]}t(H',U)\label{eq:epsU}
\end{equation}
for every $p\in (0,1)$, see~\cite{Lov11,Sid89} and also~\cite[proof of Proposition 16.27]{Lov12}.
In particular, the $k$-term in \eqref{eq:epsU} depends on the $k$-deck of $H$ and the kernel $U$ only,
which we discuss in more detail in Subsection~\ref{subsec:prelim3}.

The next proposition is implied by~\cite[Equation~(7.22)]{Lov12}.
In particular, if $U$ is a kernel, then $t(C_k,U)=0$ if and only if $U$ is zero.
\begin{proposition}
\label{prop:cycle}
It holds that $t(C_k,U)>0$ for any even cycle $C_k$ and any non-zero kernel $U$.
\end{proposition}

We will need an extension of homomorphic densities to rooted graphs.
If $H$ is a graph with a distinguished vertex $w$ (the root),
then the \emph{homomorphic density} of $H$ in $U$ is the function $t^H_U:[0,1]\to\RR$ defined as
\[t^H_U(z)=\int_{[0,1]^{V(H)\setminus\{w\}}} \prod_{uw\in E(H)} U(x_u,z) \prod_{\substack{uv\in E(H)\\u,v\not=w}} U(x_u,x_v) \dd x_{V(H)\setminus\{w\}};\]
we will omit displaying the choice of $w$ in the notation as it will always be clear from the context.
In particular, if $H$ is vertex transitive, the choice of the vertex $w$ is irrelevant, and
we can write $t^H_U$ without any danger of confusion.
If $H$ is vertex transitive, then $H\oplus P_n$
is the rooted graph obtained by identifying a vertex of $H$ with one end vertex of the path $P_n$ and
choosing the other end of the path to be the root.
A kernel $U$ can be viewed as an operator, i.e., if $f:[0,1]\to\RR$ is a measurable function,
then $Uf$ is the function defined as
\[(Uf)(z)=\int_{[0,1]}U(z,x)f(x)\dd x.\]
Observe that $t^{H\oplus P_n}_U=U^nt^H_U$, in particular, $t^{H\oplus P_1}=Ut^H_U$.

The following clearly holds for every graph $H$:
\[t(H,U)=\int_{[0,1]}t^H_U(x)\dd x;\]
the choice of the root for the definition of $t^H_U(x)$ is irrelevant for the above identity to hold.
In addition, if $H_1$ and $H_2$ are two rooted graphs and $H_1\oplus H_2$ is obtained by identifying their roots,
it holds that
\begin{equation}
t(H_1\oplus H_2,U)=\int_{[0,1]}t_U^{H_1}(x)t_U^{H_2}(x)\dd x.\label{eq:glue}
\end{equation}
In particular, if $H$ is a vertex-transitive graph, then $t(H\oplus H,U)\ge 0$ and
the equality holds if and only if $t^H_U(x)=0$ for almost every $x\in [0,1]$.
We say that a kernel $U$ is \emph{balanced} if $t_U^{P_1}(x)=0$ for almost every $x\in [0,1]$,
i.e., the perturbation determined by $U$ does not change the degrees of the vertices of a graphon.
The identity \eqref{eq:glue} implies the following.

\begin{proposition}
\label{prop:degreeone}
If a kernel $U$ is balanced and a graph $H$ has a vertex of degree one, then $t(H,U)=0$.
\end{proposition}

We conclude this subsection with the following proposition on balanced kernels.

\begin{proposition}
\label{prop:balanced}
Let $U$ be a balanced kernel.
It holds that
\[\int_{[0,1]} (Uf)(x)\dd x=0\]
for every $f\in L_2[0,1]$.
\end{proposition}

\begin{proof}
Since $U$ viewed as an operator on $L_2[0,1]$ is self-adjoint and compact (as all Hilbert-Schmidt integral operators are),
there exists a finite or countable set $I$, non-zero reals $\lambda_i$ and orthonormal functions $f_i\in L_2[0,1]$ such that
\[U(x,y)=\sum_{i\in I}\lambda_i f_i(x)f_i(y).\]
Let $h\in L_2[0,1]$ be the function equal to one everywhere on $[0,1]$.
Since $U$ is balanced, it holds that $Uh$ is equal to zero almost everywhere.
Hence, the following holds for every $i\in I$:
\[0=\int_{[0,1]} f_i(x)(Uh)(x)\dd x=\lambda_i \int_{[0,1]} f_i(x)h(x)\dd x,\]
i.e., $f_i$ and $h$ are orthogonal.
It follows that $Uf$ is orthogonal to $h$ for every $f\in L_2[0,1]$ and the proposition follows.
\end{proof}

\subsection{Perturbations}
\label{subsec:prelim3}

We next analyze the dependence of the density of $G$ in $1/2+\varepsilon U$ on a kernel $U$ and $\varepsilon$.
First observe that if $U$ is a kernel, then it holds by \eqref{eq:epsU} that
\begin{align}
& t(G,1/2+\varepsilon U)+t(G,1/2-\varepsilon U)\nonumber\\
& =2^{-\|G\|+1}+\sum_{\ell\in[\|G\|]}2^{-\|G\|+\ell}\left(\sum_{H\in G[\ell]}t(H,U)+t(H,-U)\right)\varepsilon^\ell.\label{eq:Taylor0}
\end{align}
If the number of edges of a graph $H$ is odd, then $t(H,U)=-t(H,-U)$.
In particular, the coefficients at odd powers of $\varepsilon$ in \eqref{eq:Taylor0} are equal to zero.
Hence, we set
\[c^G_{U,\ell}=\sum_{H\in G[\ell]}t(H,U)\]
for a kernel $U$, a graph $G$ and an (even) integer $\ell$, and observe that
\begin{equation}
t(G,1/2+\varepsilon U)+t(G,1/2-\varepsilon U)=2^{-\|G\|+1}\left(1+\sum_{\substack{\ell\in[\|G\|]\\ \mbox{$\ell$ even}}}2^{\ell}c^G_{U,\ell}\varepsilon^\ell\right).
\label{eq:Taylor}
\end{equation}
We emphasize that the coefficient $c^G_{U,\ell}$ depends on a kernel $U$ and the $\ell$-deck $\DD$ of $G$ only.
Hence, we define
\[c^{\DD}_{U,\ell'}=\sum_{H\in\DD}\sum_{H'\in H[\ell']}t(H',U)\]
for an $\ell$-deck $\DD$ and an even positive integer $\ell'\le\ell$.
Observe that if $\DD$ is the $\ell$-deck of a graph $G$,
then the coefficients $c^G_{U,\ell'}$ and $c^{\DD}_{U,\ell'}$ have the same sign for all $\ell'=2,4,\ldots,\ell$.
Also observe that the value of $c^{\DD}_{U,\ell}$ is determined by $s_{\DD}(H)$ for all $\ell$-edge graphs $H$,
i.e., it holds that
\begin{equation}
c^{\DD}_{U,\ell'}=\sum_{H,\|H\|=\ell'}s_{\DD}(H)t(H,U)
\label{eq:coeff}
\end{equation}
for every $\ell$-deck $\DD$ and every even positive integer $\ell'\le\ell$.

The above leads to the following classification of $\ell$-decks, which we have already mentioned in Section~\ref{sec:intro}.
Let $\ell$ be an even integer.
An $\ell$-deck $\DD$ is of 
\begin{description}
\item[Class I] 
if for every non-zero kernel $U$,
not all of the coefficients $c^{\DD}_{U,2},\ldots$, $c^{\DD}_{U,\ell}$ are zero and
the first non-zero coefficient among $c^{\DD}_{U,2},\ldots$, $c^{\DD}_{U,\ell}$ is positive;
\item[Class II]
if there exists a (non-zero) kernel $U$ such that
not all of the coefficients $c^{\DD}_{U,2},\ldots$, $c^{\DD}_{U,\ell}$ are zero and
the first non-zero coefficient among $c^{\DD}_{U,2},\ldots$, $c^{\DD}_{U,\ell}$ is negative;
\item[Class III]
if there exists a non-zero kernel $U$ such that
all the coefficients $c^{\DD}_{U,2},\ldots$, $c^{\DD}_{U,\ell}$ are zero, and
for every (non-zero) kernel $U$, it holds that
either all the coefficients $c^{\DD}_{U,2},\ldots$, $c^{\DD}_{U,\ell}$ are zero or
the first non-zero coefficient among $c^{\DD}_{U,2},\ldots$, $c^{\DD}_{U,\ell}$ is positive.
\end{description}
In particular, the following holds for every graph $G$ with $m$ edges and every $\ell\le m$:
every graph $G$ such that its $\ell$-deck is of Class I is locally common,
every graph $G$ such that its $\ell$-deck is of Class II is not locally common, and
every graph $G$ such that its $\ell$-deck is of Class III and $\ell\in\{m-1,m\}$ is locally common;
in particular our results give a full characterization of graphs with up to $13$ edges into Class I or Class II.
If the $\ell$-deck is of Class III and $\ell<m-1$, then it cannot be decided based on the $\ell$-deck whether $G$ is locally common or not;
in particular if a $12$-deck is of Class III and $m>13$, then $G$ is of Class III in the sense defined in Section~\ref{sec:intro}.
We remark that in our analysis above
we have used that the multiplicities of $\ell'$-edge graphs $H$ in the $\ell'$-deck of $G$ and in the $\ell'$-deck of $G[\ell]$
differ by the same multiplicative constant independent of $H$.
Also observe that if the $\ell'$-deck of an $\ell$-deck $\DD$, $\ell'\le\ell$, is of Class I, then $\DD$ is also of Class I, and
if the $\ell'$-deck of $\DD$ is of Class II, then $\DD$ is also of Class II.

\begin{proposition}
\label{prop:coeff2}
It holds that $c^{\DD}_{U,2}\ge 0$ for every $\ell$-deck $\DD$ and every kernel $U$.
Moreover, if $s_{\DD}(P_2)>0$, then $c^{\DD}_{U,2}=0$ if and only if the kernel $U$ is balanced.
\end{proposition}

\begin{proof}
The definition of the coefficient $c^{\DD}_{U,2}$ yields that
\[c^{\DD}_{U,2}=s_{\DD}(P_1\cup P_1)\left(\int_{[0,1]}t_U^{P_1}(x)\dd x\right)^2+s_{\DD}(P_2)\int_{[0,1]}t_U^{P_1}(x)^2\dd x.\]
It follows that the coefficient $c^{\DD}_{U,2}$ is always non-negative.

It remains to prove the second part of the proposition.
So, suppose that $s_{\DD}(P_2)>0$.
If the kernel $U$ is balanced, then both integrals above are zero.
On the other hand, if $c^{\DD}_{U,2}=0$, then the second integral above must be zero,
which is possible only if $t_U^{P_1}(x)=0$ for almost every $x\in [0,1]$.
We conclude that $c^{\DD}_{U,2}=0$ if and only if the kernel $U$ is balanced.
\end{proof}

Proposition~\ref{prop:degreeone} and the characterization results obtained in Theorems~\ref{thm:deck8}, \ref{thm:deck10} and \ref{thm:deck12}
lead us to the definition of a principal graph:
\begin{definition}
A graph $H$ is \emph{principal}
if either $H$ is an even cycle or $H$ has minimum degree two and every block of $H$ is an odd cycle or an edge.
\end{definition}
Principal graphs with four, six, eight, ten and twelve edges are listed in Figures~\ref{fig:princ4},
\ref{fig:princ6}, \ref{fig:princ8}, \ref{fig:princ10} and \ref{fig:princ12}, respectively.
For a deck $\DD$ with $s_{\DD}(P_2)>0$,
let $g$ be the length of the shortest even cycle that a graph in $\DD$ contains, and
let $\DD'$ be the $g'$-deck $\DD$, where $g'\leq g$ is even.
Then the frequencies of principal graphs of $\DD'$ determine its class.
Indeed, consider a non-zero kernel $U$.
Since $s_{\DD}(P_2)>0$, it holds that $c^U_{\DD,2}>0$ by Proposition~\ref{prop:coeff2} unless $U$ is balanced.
If $U$ is balanced, then $t(H,U)=0$ for every graph $H$ with vertex of degree one by Proposition~\ref{prop:degreeone},
i.e., the non-zero contribution to the sum \eqref{eq:coeff} defining $c^U_{\DD,4},\ldots,c^U_{\DD,g'}$
comes only from subgraphs with minimum degree two.
Since all graphs of $\DD'$ that have minimum degree two are principal,
the class of $\DD'$ is determined by its principal graphs.

\section{Sums of powers}
\label{sec:sums}

In order to present our main tool (Lemma \ref{lm:core}) in Section \ref{sec:kernel},
we need to state some number-theoretic results. 
The main result of this section is Lemma \ref{lm:powers}, the proof of which follows from 
Lemmas \ref{lm:powers-zero} and \ref{lm:powers-small}. 

\begin{lemma}
\label{lm:powers-zero}
For every pair of odd integers $k_0$ and $k$ such that $3\le k_0\le k$,
there exists an integer $m$ and reals $\omega_1,\ldots,\omega_m$ such that
\[\sum_{i\in [m]}\omega_i^{\ell}=0\]
for every odd integer $\ell\not=k_0$, $3\le\ell\le k$, and
\[\sum_{i\in [m]}\omega_i^{k_0}>0.\]
\end{lemma}

\begin{proof}
Consider the square matrix $A$ of order $(k-1)/2$ such that $A_{ij}=j^{2i+1}$ for $i,j\in [(k-1)/2]$ and
the square matrix $B$ of the same order such that $B_{ij}=j^{2i-2}$ for $i,j\in [(k-1)/2]$.
The matrix $B$ is a Vandermonde matrix and so it is full rank.
Since the matrix $A$ can be obtained from the matrix $B$ by multiplying its $j$-th column by $j^3$,
the matrix $A$ is also full rank.
It follows that there exists a rational vector $z\in\RR^{(k-1)/2}$ such that
the vector $Az$ is the $(k_0-1)/2$-th unit vector,
i.e., $(Az)_{(k_0-1)/2}=1$ and $(Az)_i=0$ for $i\not=(k_0-1)/2$.
Hence, there exists an integer vector $z'\in\ZZ^{(k-1)/2}$ such that
$(Az')_{(k_0-1)/2}>0$ and $(Az')_i=0$ for $i\not=(k_0-1)/2$.
We set $m=|z'_1|+\cdots+|z'_{(k-1)/2}|$ and
consider the multiset of $m$ reals that
contains $j$ with multiplicity $z'_j$ if $z'_j\ge 0$ and $-j$ with multiplicity $-z'_j$ if $z'_j<0$ for each $j\in [(k-1)/2]$.
Setting $\omega_1,\ldots,\omega_m$ to be the elements of this multiset yields the statement of the lemma.
\end{proof}

\begin{lemma}
\label{lm:powers-small}
For every pair of odd integers $k_0$ and $k$ such that $3\le k_0\le k$ and every positive real $\delta>0$,
there exists an integer $m$ and reals $\omega_1,\ldots,\omega_m$ such that
\[\sum_{i\in [m]}\omega_i^{\ell}=0\]
for every odd integer $\ell\not=k_0$, $3\le\ell\le k$, and
\[\sum_{i\in [m]}\omega_i^{k_0}=1\qquad\mbox{and}\qquad\sum_{i\in [m]}\omega_i^{k+1}\le\delta.\]
\end{lemma}

\begin{proof}
Apply Lemma~\ref{lm:powers-zero} with $k_0$ and $k$ to obtain $\omega_1,\ldots,\omega_{m'}$ such that
\[\sum_{i\in [m']}\omega_i^{\ell}=0\]
for every odd integer $\ell\not=k_0$, $3\le\ell\le k$, and
\[\sum_{i\in [m']}\omega_i^{k_0}=\Omega>0.\]
For an integer $n\in N$, consider a multiset $A$ of $m=2^{nk_0}m'$ numbers that
contains each of the numbers $\omega_i\Omega^{-1/k_0}2^{-n}$, $i\in [m']$, with multiplicity $2^{nk_0}$.
Clearly, the sum of the $k_0$-th powers of the numbers in $A$ is equal to one and
the sum of the $\ell$-th powers for odd $\ell\not=k_0$, $3\le\ell\le k$, is equal to zero.
The sum of the $(k+1)$-th powers can be bounded as follows:
\[\sum_{\omega\in A}\omega^{k+1}=\frac{2^{nk_0}}{\Omega^{(k+1)/k_0}2^{n(k+1)}}\sum_{i\in [m']}\omega_i^{k+1}\le \frac{1}{\Omega^{(k+1)/k_0}2^n}\sum_{i\in [m']}\omega_i^{k+1}.\]
Hence, there exists $n\in\NN$ such that the sum is at most $\delta$ and
the lemma holds for the multiset $A$ for such a choice of $n$.
\end{proof}

\begin{lemma}
\label{lm:powers}
For every odd integer $k\ge 3$, all reals $s_3,s_5,\ldots,s_k$ and every positive real $\delta>0$,
there exists an integer $m$ and reals $\omega_1,\ldots,\omega_m$ such that
\[\sum_{i\in [m]}\omega_i^{\ell}=s_{\ell}\]
for every odd integer $\ell$, $3\le\ell\le k$, and
\[\sum_{i\in [m]}\omega_i^{k+1}\le\delta.\]
\end{lemma}

\begin{proof}
For $\ell=3,5,\ldots,k$,
if $s_{\ell}\not=0$,
we apply Lemma~\ref{lm:powers-small} with $\frac{\delta}{ks_{\ell}^{(k+1)/\ell}}$
to get $m_{\ell}$ and $\omega_{\ell,1},\ldots,\omega_{\ell,m_{\ell}}$ such that
\[\sum_{i\in [m_{\ell}]}\omega_{\ell,i}^{j}=0\]
for every odd integer $j\not=\ell$, $3\le j \le k$, and
\[\sum_{i\in [m_{\ell}]}\omega_{\ell,i}^{\ell}=1\qquad\mbox{and}\qquad\sum_{i\in [m_{\ell}]}\omega_{\ell,i}^{k+1}\le\frac{\delta}{ks_{\ell}^{(k+1)/\ell}}.\]
If $s_{\ell}=0$, we set $m_{\ell}=0$.

We set $m=m_3+\cdots+m_k$ and
consider the multiset of $m$ reals $\omega_1,\ldots,\omega_m$ that
consists of $\omega_{\ell,i}s_{\ell}^{1/\ell}$ for $\ell=3,\ldots,k$ and $i\in [m_\ell]$.
Observe that for every $\ell=3,\ldots,k$, it holds that
\[\sum_{i\in [m]}\omega_i^{\ell}=\sum_{i\in [m_\ell]}\left(\omega_{\ell,i}s_{\ell}^{1/\ell}\right)^\ell
                                =s_\ell\sum_{i\in [m_\ell]}\omega_{\ell,i}^\ell=s_\ell.\]
In addition, it holds that
\[\sum_{i\in [m]}\omega_i^{k+1}=\sum_{\ell=3,5,\ldots,k}s_{\ell}^{(k+1)/\ell}\sum_{i\in [m_\ell]}\omega_{\ell,i}^{k+1}
                               \le\sum_{\ell=3,5,\ldots,k}\frac{\delta}{k}\le\delta.\]
The lemma follows.
\end{proof}

\section{Constructing balanced perturbations}
\label{sec:kernel}

In this section we construct a special family of perturbations and
determine the corresponding densities of key principal graphs.
The construction is presented in the following lemma.

\begin{lemma}
\label{lm:core}
Let $k\ge 3$ be an odd integer, let $\delta\in (0,1)$ be a positive real, let $m$ be a non-negative integer, and
let $\sigma_i$, $\gamma_{\ell}$ and $\tau_{i,\ell}$, $i\in [m]$ and $\ell=3,5,\ldots,k$, be any reals
such that $\sigma_1^{k+1}+\cdots+\sigma_m^{k+1}\le\delta/2$.
There exists a non-zero balanced kernel $U$,
orthonormal functions $f_1,\ldots,f_m\in L_2[0,1]$ and $g_3,\ldots,g_k\in L_2[0,1]$ and
a real $\gamma>0$ such that
$f_i$, $i\in [m]$, is an eigenfunction of $U$ associated with $\sigma_i$, i.e., $Uf_i=\sigma_if_i$, and
\[\int_{[0,1]}f_i(x)\dd x=0\]
for every $i\in [m]$,
the functions $g_\ell$, $\ell=3,5,\ldots,k$, belong to the kernel of $U$, and
\[\int_{[0,1]}g_{\ell}(x)\dd x=\gamma\]
for every $\ell=3,5,\ldots,k$, and
\[t^{C_\ell}_U(x)=\frac{\gamma_\ell}{\gamma} g_\ell(x)+\sum_{i\in [m]}\tau_{i,\ell}f_i(x)\]
for every odd integer $\ell$, $3\le\ell\le k$, and $t(C_{k+1},U)\le\delta$.
\end{lemma}

\begin{figure}
\begin{center}
\epsfbox{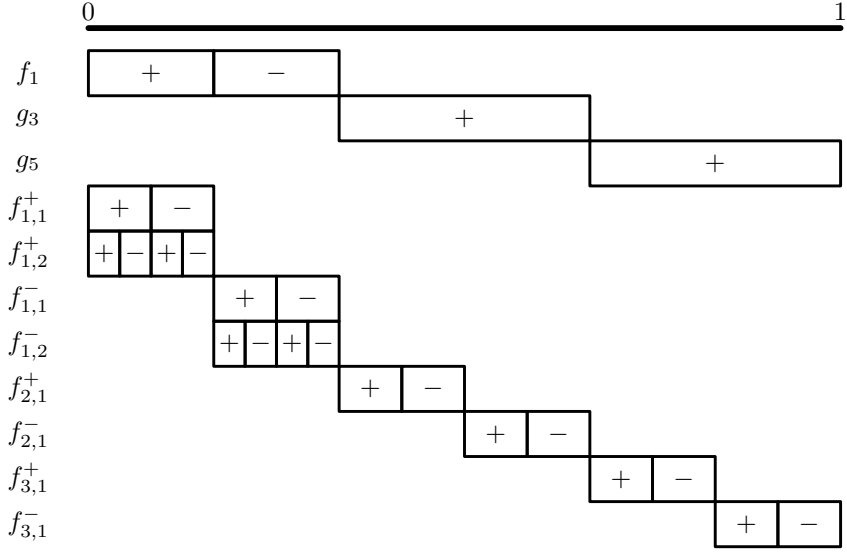}
\end{center}
\caption{The signs of the functions from the proof of Lemma~\ref{lm:core} when $k=5$ and $m=1$.}
\label{fig:core}
\end{figure}

\begin{proof}
We can assume that at least one of the values $\sigma_i$, $\gamma_{\ell}$ and $\tau_{i,\ell}$ is non-zero;
if this is not the case, we will prove the lemma for $k+2$ and
the additional values set as $\gamma_{k+2}=\delta/2$ and $\tau_{i,k+2}=0$, $i\in [m]$.
Let $\alpha_i=\frac{i}{m+(k-1)/2}$ for $i=0,\ldots,m+(k-1)/2$;
note that the intervals $[\alpha_{i-1},\alpha_i)$, $i\in [m+(k-1)/2]$, partition the interval $[0,1)$.
For $\ell=3,5,\ldots,k$, we define
\[g_{\ell}(x)=\begin{cases}
              (m+(k-1)/2)^{1/2} & \mbox{if $\alpha_{m+(\ell-3)/2}\le x<\alpha_{m+(\ell-1)/2}$,} \\
	      0 & \mbox{otherwise.}
	      \end{cases}\]
For $i=1,\ldots,m$, we define
\[f_i(x)=\begin{cases}
         (m+(k-1)/2)^{1/2} & \mbox{if $\alpha_{i-1}\le x<\frac{\alpha_{i-1}+\alpha_i}{2}$,} \\
	 -(m+(k-1)/2)^{1/2} & \mbox{if $\frac{\alpha_{i-1}+\alpha_i}{2}\le x<\alpha_i$,} \\
	 0 & \mbox{otherwise.}
	 \end{cases}\]
Note that the functions $g_3,\ldots,g_k$ and $f_1,\ldots,f_m$ form an orthonormal system of functions such that
\[\int_{[0,1]}g_{\ell}(x)\dd x=\gamma\mbox{ and }\int_{[0,1]}f_i(x)\dd x=0\]
for every $\ell=3,5,\ldots,k$ and for every $i\in [m]$, where $\gamma=(m+(k-1)/2)^{-1/2}$.

We next construct the kernel $U$.
We start with defining functions $h_j:[0,1]\to\RR$, $j\in\NN$, as
\[h_j(x)=\begin{cases}
         +(2m+k-1)^{1/2} & \mbox{if $\lfloor 2^jx\rfloor$ is even, and}\\
	 -(2m+k-1)^{1/2} & \mbox{otherwise.}
	 \end{cases}\]
For $i=1,\ldots,m+(k-1)/2$ and $j\in\NN$,
we define a function $f_{i,j}^+:[0,1]\to\RR$ as
\[f_{i,j}^+(x)=\begin{cases}
               h_j\left((2m+k-1)(x-\alpha_{i-1})\right) & \mbox{if $\alpha_{i-1}\le x<\frac{\alpha_{i-1}+\alpha_i}{2}$, and}\\
	       0 & \mbox{otherwise,}
	       \end{cases}\]
and a function $f_{i,j}^-:[0,1]\to\RR$ as
\[f_{i,j}^-(x)=\begin{cases}
               h_j\left((2m+k-1)\left(x-\frac{\alpha_{i-1}+\alpha_i}{2}\right)\right) & \mbox{if $\frac{\alpha_{i-1}+\alpha_i}{2}\le x<\alpha_i$, and}\\
	       0 & \mbox{otherwise.}
	       \end{cases}\]
See Figure~\ref{fig:core} for an illustration.

For $i=1,\ldots,m$,
we apply Lemma~\ref{lm:powers} with $k$, $s_j=\frac{\tau_{i,j}}{2(m+(k-1)/2)^{1/2}}-\frac{\sigma_i^j}{2}$ for $j=3,\ldots,k$, and $\frac{\delta}{2(2m+k-1)}$,
to get $\omega^+_{i,1},\ldots,\omega^+_{i,m^+_i}$ such that
the sum of their $j$-th powers is equal to $s_j=\frac{\tau_{i,j}}{2(m+(k-1)/2)^{1/2}}-\frac{\sigma_i^j}{2}$ and
the sum of their $(k+1)$-th powers is at most $\frac{\delta}{2(2m+k-1)}$.
We next apply Lemma~\ref{lm:powers} with $k$, $s_j=\frac{-\tau_{i,j}}{2(m+(k-1)/2)^{1/2}}-\frac{\sigma_i^j}{2}$ for $j=3,\ldots,k$, and $\frac{\delta}{2(2m+k-1)}$,
to get $\omega^-_{i,1},\ldots,\omega^-_{i,m^-_i}$ such that
the sum of their $j$-th powers is equal to $s_j=\frac{-\tau_{i,j}}{2(m+(k-1)/2)^{1/2}}-\frac{\sigma_i^j}{2}$ and
the sum of their $(k+1)$-th powers is at most $\frac{\delta}{2(2m+k-1)}$.
For $i=1,\ldots,(k-1)/2$,
we apply Lemma~\ref{lm:powers} with $k$, $s_{2i+1}=\frac{\gamma_{2i+1}}{2\gamma(m+(k-1)/2)^{1/2}}$ and $s_j=0$ for $j\not=2i+1$, and $\frac{\delta}{2(2m+k-1)}$,
to get $\omega_{m+i,1}\ldots,\omega_{m+i,m_{m+i}}$ such that
the sum of their $j$-th powers is equal to $0$ unless $j=2i+1$ and it is equal to $\frac{\gamma_{2i+1}}{2\gamma(m+(k-1)/2)^{1/2}}$ if $j=2i+1$, and
the sum of their $(k+1)$-th powers is at most $\frac{\delta}{2(2m+k-1)}$.
We define the kernel $U$ as
\begin{align*}
U(x,y) &=\sum_{i\in [m]}\sigma_{i}f_i(x)f_i(y)+\\
       &+\sum_{i\in [m]}\sum_{j\in [m_i^+]}\omega_{i,j}^+f_{i,j}^+(x)f_{i,j}^+(y)+\sum_{i\in [m]}\sum_{j\in [m_i^-]}\omega_{i,j}^-f_{i,j}^-(x)f_{i,j}^-(y)\\
       &+\sum_{i\in [(k-1)/2]}\sum_{j\in [m_{m+i}]}\omega_{m+i,j}(f_{m+i,j}^+(x)f_{m+i,j}^+(y)-f_{m+i,j}^-(x)f_{m+i,j}^-(y)).
\end{align*}       
Since the integral of each of the functions $f_i$ for $i\in [m]$,
$f_{i,j}^+$ for $i\in [m]$ and $j\in [m_i^+]$, 
$f_{i,j}^-$ for $i\in [m]$ and $j\in [m_i^-]$, and
$f_{m+i,j}^+$ and $f_{m+i,j}^-$ for $i\in [(k-1)/2]$ and $j\in [m_{m+i}]$ over $[0,1]$ is zero,
it follows that the kernel $U$ is balanced.
Next observe that
\begin{align*}
t(C_{k+1},U) &= \sum_{i\in [m]}\sigma_{i}^{k+1}+\sum_{i\in [m]}\sum_{j\in [m_i]}\left(\omega_{i,j}^+\right)^{k+1}+\sum_{i\in [m]}\sum_{j\in [m_i^-]}\left(\omega_{i,j}^-\right)^{k+1}\\
             &+2\sum_{i\in [(k-1)/2]}\sum_{j\in [m_{m+i}]}\omega_{m+i,j}^{k+1}\\
             &\le \frac{\delta}{2}+m\cdot\frac{\delta}{2m+k-1}+\frac{k-1}{2}\cdot\frac{\delta}{2m+k-1}=\delta.
\end{align*}
For $\ell=3,\ldots,k$, we obtain that
\begin{align*}
t^{C_\ell}_U(x) & = \sum_{i\in [m]}\sigma_{i}^\ell f_i(x)^2+
                    \sum_{i\in [m]}\sum_{j\in [m_i]}\left(\omega_{i,j}^+\right)^\ell f_{i,j}^+(x)^2+\sum_{i\in [m]}\sum_{j\in [m_i^-]}\left(\omega_{i,j}^-\right)^\ell f_{i,j}^-(x)^2\\
		& +\sum_{i\in [(k-1)/2]}\sum_{j\in [m_{m+i}]}\omega_{m+i,j}^\ell \left(f_{m+i,j}^+(x)^2+f_{m+i,j}^-(x)^2\right)\\
		& = \sum_{i\in [m]}\tau_{i,\ell}f_i(x)+\frac{\gamma_{\ell}}{\gamma}g_{\ell}(x).
\end{align*}
This concludes the proof of the lemma.
\end{proof}

The next lemma summarizes key properties of kernels obtained by applying Lemma~\ref{lm:core}.

\begin{lemma}
\label{lm:core-apply}
Let $U$ be the kernel obtained by applying Lemma~\ref{lm:core} with $k$, $\delta$, $m$,
$\gamma$, $\sigma_i$, $\gamma_\ell$ and $\tau_{i,\ell}$, with $i\in [m]$ and $\ell=3,5,\ldots,k$.
It holds that 
\[t(C_{\ell},U)=\gamma_{\ell}\qquad\mbox{and}\qquad t(C_\ell\oplus C_{\ell},U)=\frac{\gamma_{\ell}^2}{\gamma^2}+\sum_{i\in [m]}\tau_{i,\ell}^2\]
for every $\ell=3,5,\ldots,k$.
Moreover, if $\ell$ and $\ell'$ are odd integers between $3$ and $k$ and $n$ is a non-negative integer such that
$\ell\not=\ell'$ or $n>0$, then it holds that
\[t(C_\ell\oplus P_n\oplus C_{\ell'},U)=\sum_{i\in [m]}\sigma_i^n\tau_{i,\ell}\tau_{i,\ell'};\]
we interpret $0^0$ in the sum above as $1$.
\end{lemma}

An immediate corollary of Lemmas~\ref{lm:core} and~\ref{lm:core-apply} is the following.

\begin{lemma}
\label{lm:core-cycle}
For every odd integer $k\ge 3$, there exists a non-zero balanced kernel $U$ such that
$t^{C_\ell}_U(x)=0$ for every odd integer $\ell$, $3\le\ell\le k$, and every $x\in [0,1]$.
In particular, $t(C_\ell,U)=0$ for every odd integer $\ell$, $3\le\ell\le k$.
\end{lemma}

A less straightforward corollary of Lemma~\ref{lm:core} is the following.

\begin{lemma}
\label{lm:core-III}
Let $\DD$ be a deck and $k$ an even integer such that no graph in $\DD$ contains an even cycle of length at most $k$.
There exists a non-zero kernel $U$ such that $c^U_{\DD,2}=\cdots=c^U_{\DD,k}=0$.
\end{lemma}

\begin{proof}
Apply Lemma~\ref{lm:core} with $k-1$, $\delta=1$, $m=0$ and $\gamma_3=\cdots=\gamma_{k-1}=0$
to get a non-zero balanced kernel $U$ with the properties given in the statements of Lemmas~\ref{lm:core} and~\ref{lm:core-apply}.
Since $U$ is balanced, it follows that $c^U_{\DD,2}=0$.
Since $t^{C_{\ell}}_U(x)=0$ for every $\ell=3,\ldots,k-1$ and almost every $x\in [0,1]$ and
every principal graph $H$ with at most $k$ edges contains an end-block that is an odd cycle,
it holds that $t(H,U)=0$ for every principal graph $H$.
Hence, it also holds that $c^U_{\DD,4}=\cdots=c^U_{\DD,k}=0$.
\end{proof}

\section{Decks with at most eight edges}
\label{sec:deck8}

In this section we prove Theorem~\ref{thm:deck8}, 
which determines which class a deck of size up to 8 belongs to;
the characterization is visualized in Figure~\ref{fig:deck8} and Table~\ref{tab:deck8}.
We start with Lemmas~\ref{lm:deck4}~and~\ref{lm:deck6},
which deal with decks of size 4 and 6, respectively.

\begin{figure}
\begin{center}
\epsfbox{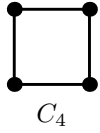}
\end{center}
\caption{The only principal $4$-edge graph.}
\label{fig:princ4}
\end{figure}

\begin{lemma}
\label{lm:deck4}
A $4$-deck $\DD$ with $s_{\DD}(P_2)>0$ is of Class I if and only if $s_{\DD}(C_4)>0$;
otherwise, $\DD$ is of Class III.
\end{lemma}

\begin{proof}
Fix a $4$-deck $\DD$ with $s_{\DD}(P_2)>0$.
Assume that $s_{\DD}(C_4)>0$ and consider a non-zero kernel $U$.
If $U$ is not balanced, then $c^U_{\DD,2}>0$ by Proposition~\ref{prop:coeff2}.
If $U$ is balanced, then $c^U_{\DD,2}=0$ by Proposition~\ref{prop:coeff2} and
$c^U_{\DD,4}=s_{\DD}(C_4)t(C_4,U)>0$ by Proposition~\ref{prop:cycle}.
It follows that the deck $\DD$ is of Class I.

Assume that $s_{\DD}(C_4)=0$ and consider any non-zero kernel $U$.
It follows that $c^U_{\DD,2}\ge 0$ by Proposition~\ref{prop:coeff2} and the equality holds only if the kernel $U$ is balanced.
If the kernel $U$ is balanced, then $t(H,U)=0$ for any $4$-edge graph that is not principal and
thus $c^U_{\DD,4}=s_{\DD}(C_4)t(C_4,U)=0$.
Lemma~\ref{lm:core-III} implies the existence of a non-zero kernel $U$ such that $c^U_{\DD,2}=c^U_{\DD,4}=0$,
which implies that the deck $\DD$ is of Class III.
\end{proof}

\begin{figure}
\begin{center}
\epsfbox{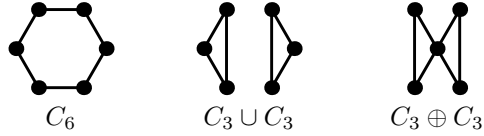}
\end{center}
\caption{Principal $6$-edge graphs.}
\label{fig:princ6}
\end{figure}

\begin{lemma}
\label{lm:deck6}
A $6$-deck $\DD$ with $s_{\DD}(P_2)>0$ is of Class I if and only if $s_{\DD}(C_4)>0$ or $s_{\DD}(C_6)>0$;
otherwise, $\DD$ is of Class III.
\end{lemma}

\begin{proof}
Fix a $6$-deck $\DD$ with $s_{\DD}(P_2)>0$.
If the $4$-deck of $\DD$ is of Class I, then the $6$-deck $\DD$ is also of Class I.
Hence, we can assume that $s_{\DD}(C_4)=0$ by Lemma~\ref{lm:deck4}.

We first consider the case that $s_{\DD}(C_6)>0$.
Consider a non-zero kernel $U$.
If $U$ is not balanced, then $c^U_{\DD,2}>0$ by Proposition~\ref{prop:coeff2}.
If $U$ is balanced, then $c^U_{\DD,2}=0$ by Proposition~\ref{prop:coeff2} and $t(H,U)=0$
for any non-principal $4$-edge or $6$-edge graph $H$;
the three principal $6$-edge graphs are listed in Figure~\ref{fig:princ6}.
Hence, it holds that $c^U_{\DD,4}=s_{\DD}(C_4)t(C_4,U)=0$ and
\[c^U_{\DD,6}=s_{\DD}(C_6)t(C_6,U)+s_{\DD}(C_3\oplus C_3)t(C_3\oplus C_3,U)+s_{\DD}(C_3\cup C_3)t(C_3\cup C_3,U).\]
Since $t(C_3\cup C_3,U)=t(C_3,U)^2\ge 0$ and
\[t(C_3\oplus C_3,U)=\int_{[0,1]}t^{C_3}_U(x)^2\dd x\ge 0,\]
it follows that $c^U_{\DD,6}\ge s_{\DD}(C_6)t(C_6,U)$,
which is positive by Proposition~\ref{prop:cycle}.
We conclude that the deck $\DD$ is of Class I.

It remains to consider the case that $s_{\DD}(C_6)=0$.
We first show that there is no kernel $U$ such that
the first non-zero coefficient among $c^U_{\DD,2}$, $c^U_{\DD,4}$ and $c^U_{\DD,6}$ (if such a coefficient exists) is negative.
Let $U$ be a kernel.
If $U$ is not balanced, then $c^U_{\DD,2}>0$ by Proposition~\ref{prop:coeff2}, and
otherwise, $c^U_{\DD,2}=0$ and $t(H,U)=0$ for any non-principal $4$-edge or $6$-edge graph $H$.
Hence, it holds that $c^U_{\DD,4}=s_{\DD}(C_4)t(C_4,U)=0$ and
\[c^U_{\DD,6}=s_{\DD}(C_3\oplus C_3)t(C_3\oplus C_3,U)+s_{\DD}(C_3\cup C_3)t(C_3\cup C_3,U)\ge 0.\]
The existence of a non-zero kernel $U$ such that $c^U_{\DD,2}=c^U_{\DD,4}=c^U_{\DD,6}=0$
follows from Lemma~\ref{lm:core-III} applied with $k=6$.
We conclude that the deck $\DD$ is of Class III.
\end{proof}

\begin{figure}
\begin{center}
\epsfbox{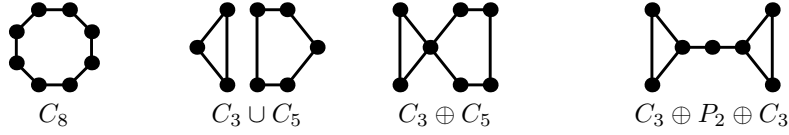}
\end{center}
\caption{Principal $8$-edge graphs.}
\label{fig:princ8}
\end{figure}

Lemmas \ref{lm:deck8-III}, \ref{lm:deck8-I}, and \ref{lm:deck8-II} describe when a given $8$-deck is of 
Class III, Class I, or Class II, respectively.

\begin{lemma}
\label{lm:deck8-III}
An $8$-deck $\DD$ with $s_{\DD}(P_2)>0$ is of Class III if the $6$-deck of $\DD$ is of Class III,
$s_{\DD}(C_8)=0$ and at least one of the following holds:
\begin{itemize}
\item $s_{\DD}(C_3\oplus C_3)>0$,
\item $s_{\DD}(C_3\oplus C_5)=0$, $s_{\DD}(C_3\cup C_3)>0$, or
\item $s_{\DD}(C_3\oplus C_5)=s_{\DD}(C_3\cup C_5)=0$.
\end{itemize}
\end{lemma}

\begin{proof}
Fix an $8$-deck $\DD$ with $s_{\DD}(P_2)>0$ such that its $6$-deck is of Class III and
that satisfies the assumption of the lemma,
i.e., $s_{\DD}(C_8)=0$ and (at least) one of the three choices in the statement of the lemma holds.
Note that $s_{\DD}(C_4)=s_{\DD}(C_6)=0$ by Lemmas~\ref{lm:deck4} and~\ref{lm:deck6}.
Lemma~\ref{lm:core-III} applied with $k=8$ yields that there exists a non-zero balanced kernel $U$ such that
$c^U_{\DD,2}=c^U_{\DD,4}=c^U_{\DD,6}=c^U_{\DD,8}=0$.
This implies that the $8$-deck $\DD$ is not of Class I. 

To establish the statement of the lemma,
we fix a kernel $U$ and
show that the first non-zero coefficient among $c^U_{\DD,2}$, $c^U_{\DD,4}$, $c^U_{\DD,6}$ and $c^U_{\DD,8}$ is positive or
does not exist.
If $U$ is not balanced, then $c^U_{\DD,2}>0$ by Proposition~\ref{prop:coeff2}.
So, we will assume that $U$ is balanced,
which implies that $t(H,U)=0$ for any graph $H$ with at most eight edges that is not principal.
In particular, it holds that $c^U_{\DD,2}=c^U_{\DD,4}=0$ and
\[c^U_{\DD,6}=s_{\DD}(C_3\oplus C_3)\int_{[0,1]}t^{C_3}_U(x)^2\dd x+s_{\DD}(C_3\cup C_3)t(C_3,U)^2.\]
Note that the coefficient $c^U_{\DD,6}$ is always non-negative.

If $s_{\DD}(C_3\oplus C_3)>0$,
then $c^U_{\DD,6}>0$ unless $t^{C_3}_U(x)=0$ for almost every $x\in [0,1]$.
However, if $t^{C_3}_U(x)=0$ for almost every $x\in [0,1]$,
we obtain that $t(C_3\cup C_5,U)=t(C_3\oplus C_5,U)=t(C_3\oplus P_2\oplus C_3)=0$,
i.e., the densities of all principal $8$-edge graphs in $U$ are zero with the exception of $C_8$.
It follows that $c^U_{\DD,8}=0$.

We next consider the case when $s_{\DD}(C_3\oplus C_5)=0$ and $s_{\DD}(C_3\cup C_3)>0$.
Observe that $c^U_{\DD,6}>0$ unless $t(C_3,U)=0$,
in which case 
$t(C_3\cup C_5,U)=t(C_3,U)t(C_5,U)=0$.
It follows that
\[c^U_{\DD,8}=s_{\DD}(C_3\oplus P_2\oplus C_3)\int_{[0,1]}t^{C_3\oplus P_1}_U(x)^2\dd x\ge 0.\]
We conclude that if the coefficient $c^U_{\DD,6}$ is zero, then $c^U_{\DD,8}$ is non-negative.

The final case given in the statement is that $s_{\DD}(C_3\oplus C_5)=s_{\DD}(C_3\cup C_5)=0$.
We obtain that
\[c^U_{\DD,8}=s_{\DD}(C_3\oplus P_2\oplus C_3)\int_{[0,1]}t^{C_3\oplus P_1}_U(x)^2\dd x\ge 0\]
without any assumptions on the coefficient $c^U_{\DD,6}$.
In particular, the coefficient $c^U_{\DD,8}$ is always non-negative.
Hence, the first non-zero coefficient, if it exists, is either $c^U_{\DD,6}$ or $c^U_{\DD,8}$ and is positive.
This concludes the proof of the lemma.
\end{proof}

\begin{lemma}
\label{lm:deck8-I}
An $8$-deck $\DD$ with $s_{\DD}(P_2)>0$ is of Class I if either the $6$-deck of $\DD$ is of Class I, or
the $6$-deck of $\DD$ is of Class III, $s_{\DD}(C_8)>0$ and at least one of the following holds:
\begin{itemize}
\item $s_{\DD}(C_3\oplus C_3)>0$,
\item $s_{\DD}(C_3\oplus C_5)=0$, $s_{\DD}(C_3\cup C_3)>0$, or
\item $s_{\DD}(C_3\oplus C_5)=s_{\DD}(C_3\cup C_5)=0$.
\end{itemize}
\end{lemma}

\begin{proof}
If the $6$-deck of $\DD$ is of Class I, then the $8$-deck $\DD$ is also of Class I.
We next assume that the $6$-deck of $\DD$ is of Class III.
Let $\DD'$ be the $8$-deck obtained from $\DD$ by removing all cycles of length eight.
Since the $8$-deck $\DD'$ satisfies the assumptions of Lemma~\ref{lm:deck8-III},
the $8$-deck $\DD'$ is of Class III.
It follows that for any non-zero kernel $U$,
all the coefficients $c^U_{\DD',2},\ldots,c^U_{\DD',8}$ are zero or
the first non-zero among these coefficients is positive.
Since $c^U_{\DD,\ell}=c^U_{\DD',\ell}$ for $\ell=2,4,6$ and $c^U_{\DD,8}=c^U_{\DD',8}+s_{\DD}(C_8)t(C_8,U)$,
we obtain using Proposition~\ref{prop:cycle} that
at least one of the coefficients $c^U_{\DD,2},\ldots,c^U_{\DD,8}$ is non-zero and
the first non-zero among these coefficients is positive.
We conclude that the $8$-deck $\DD$ is of Class I.
\end{proof}

\begin{lemma}
\label{lm:deck8-II}
An $8$-deck $\DD$ with $s_{\DD}(P_2)>0$ is of Class II if the $6$-deck of $\DD$ is of Class III and
at least one of the following holds:
\begin{itemize}
\item $s_{\DD}(C_3\oplus C_3)=0$ and $s_{\DD}(C_3\oplus C_5)>0$, or
\item $s_{\DD}(C_3\oplus C_3)=s_{\DD}(C_3\cup C_3)=0$ and $s_{\DD}(C_3\cup C_5)>0$.
\end{itemize}
\end{lemma}

\begin{proof}
For each of the two cases described in the statement of the lemma,
we will find a non-zero kernel $U$ such that 
not all of the coefficients $c^U_{\DD,2},\ldots,c^U_{\DD,8}$ are zero and
the first non-zero coefficient among them is negative.
Since the $6$-deck of $\DD$ is of Class III,
we obtain that $s_{\DD}(C_4)=s_{\DD}(C_6)=0$ by Lemma~\ref{lm:deck6}.

If $s_{\DD}(C_3\oplus C_3)=0$ and $s_{\DD}(C_3\oplus C_5)>0$,
apply Lemma~\ref{lm:core} with $k=7$, $m=1$, any $\delta\in (0,1)$ such that $s_{\DD}(C_8)\delta\le 1/2$,
$\gamma_3=\gamma_5=\gamma_7=0$, $\sigma_1=\tau_{1,7}=0$, 
$\tau_{1,3}=1$ and $\tau_{1,5}=-1$ to get a balanced kernel $U$
with the properties given in the statement of Lemma~\ref{lm:core}.
Lemma~\ref{lm:core-apply} yields that $t(C_3,U)=t(C_3\oplus P_2\oplus C_3,U)=0$,
$t(C_3\oplus C_5,U)=-1$ and $t(C_8,U)\le \delta$.
Hence, we obtain that $c^U_{\DD,4}=0$, $c^U_{\DD,6}=0$ and
\[c^U_{\DD,8}=s_{\DD}(C_8)t(C_8,U)+s_{\DD}(C_3\oplus C_5)t(C_3\oplus C_5,U)\le -1/2.\]
We conclude that $\DD$ is of Class II.

We next consider the case when $s_{\DD}(C_3\oplus C_3)=s_{\DD}(C_3\cup C_3)=0$ and $s_{\DD}(C_3\cup C_5)>0$.
We apply Lemma~\ref{lm:core} with $k=7$, $m=0$, any $\delta\in (0,1)$ such that $s_{\DD}(C_8)\delta\le 1/2$, $\gamma_3=1$, $\gamma_5=-1$, and $\gamma_7=0$.
Observe that $t(C_3\cup C_5,U)=t(C_3,U)t(C_5,U)=-1$ and $t(C_3\oplus C_5,U)=t(C_3\oplus P_2\oplus C_3,U)=0$.
Hence, the coefficients $c^U_{\DD,4}=0$ and $c^U_{\DD,6}=0$, and
\[c^U_{\DD,8}=s_{\DD}(C_8)t(C_8,U)+s_{\DD}(C_3\cup C_5)t(C_3\cup C_5,U)\le -1/2.\]
We conclude that $\DD$ is of Class II in this case, too.
\end{proof}

Lemmas~\ref{lm:deck4}--\ref{lm:deck8-II} imply the following theorem.
In addition to the diagram in Figure~\ref{fig:deck8},
we also provide the classification in Table~\ref{tab:deck8}.

\begin{theorem}
\label{thm:deck8}
Let $\DD$ be an $8$-deck with $s_{\DD}(P_2)>0$.
The deck $\DD$ is of Class I, Class II or Class III as determined
in the diagram in Figure~\ref{fig:deck8}.
\end{theorem}

\begin{proof}
The proof follows by inspecting the diagram in Figure~\ref{fig:deck8} and verifying that
every path leading to the label Class I corresponds to the assumptions of Lemma~\ref{lm:deck8-I},
every path leading to the label Class II corresponds to the assumptions of Lemma~\ref{lm:deck8-II}, and
every path leading to the label Class III corresponds to the assumptions of Lemma~\ref{lm:deck8-III}.
\end{proof}

\begin{table}
\begin{center}
\scalebox{0.9}{
\begin{tabular}{|c|c|ccc|cccc|}
\hline
$s_{\DD}(\cdot)$ & $C_4$ & $C_6$ & $C_3\cup C_3$ & $C_3\oplus C_3$ & $C_8$ & $C_3\cup C_5$ & $C_3\oplus C_5$ & $C_3\oplus P_2\oplus C_3$\\
\hline
Class I   & $>0$    & $\star$ & $\star$ & $\star$    & $\star$ & $\star$ & $\star$ & $\star$ \\
Class I   & $0$     &  $>0$   & $\star$ & $\star$    & $\star$ & $\star$ & $\star$ & $\star$ \\
Class I   & $0$     &  $0$    & $\star$ &  $>0$      &  $>0$   & $\star$ & $\star$ & $\star$ \\
Class II  & $0$     &  $0$    & $\star$ &  $0$       &  $>0$   & $\star$ &  $>0$   & $\star$ \\
Class I   & $0$     &  $0$    &  $>0$   &  $0$       &  $>0$   & $\star$ &  $0$    & $\star$ \\
Class II  & $0$     &  $0$    &  $0$    &  $0$       &  $>0$   &  $>0$   &  $0$    & $(0)$   \\
Class I   & $0$     &  $0$    &  $0$    &  $0$       &  $>0$   &  $0$    &  $0$    & $(0)$   \\
Class III & $0$     &  $0$    & $\star$ &  $>0$      &  $0$    & $\star$ & $\star$ & $\star$ \\
Class II  & $0$     &  $0$    & $\star$ &  $0$       &  $0$    & $\star$ &  $>0$   & $\star$ \\
Class III & $0$     &  $0$    &  $>0$   &  $0$       &  $0$    & $\star$ &  $0$    & $\star$ \\
Class II  & $0$     &  $0$    &  $0$    &  $0$       &  $0$    &  $>0$   &  $0$    & $(0)$   \\
Class III & $0$     &  $0$    &  $0$    &  $0$       &  $0$    &  $0$    &  $0$    & $(0)$   \\
\hline
\end{tabular}}
\end{center}
\caption{The classification of $8$-decks $\DD$ with $s_{\DD}(P_2)>0$. An entry $\star$ means arbitrary multiplicity, and
         an entry $(0)$ represents that other columns of the same row imply that $s_{\DD}(\cdot)$ is $0$.}
\label{tab:deck8}
\end{table}

\begin{figure}
\begin{center}
\epsfbox{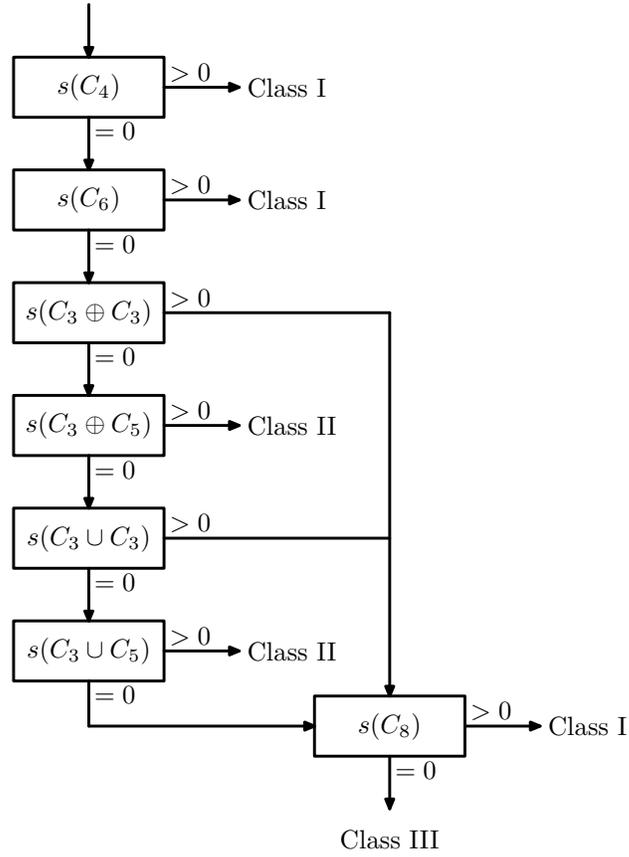}
\end{center}
\caption{The classification of $8$-decks $\DD$ with $s_{\DD}(P_2)>0$; we omit the subscript $\DD$ in the diagram.}
\label{fig:deck8}
\end{figure}

Theorem~\ref{thm:deck8} yields the following corollary.

\begin{corollary}
\label{cor:deck8-III}
Every $8$-deck $\DD$ with $s_{\DD}(P_2)>0$ that is of Class III satisfies that $s_{\DD}(C_4)=s_{\DD}(C_6)=s_{\DD}(C_8)=0$ and
exactly one of the following:
\begin{itemize}
\item $s_{\DD}(C_3\oplus C_3)>0$,
\item $s_{\DD}(C_3\oplus C_3)=s_{\DD}(C_3\oplus C_5)=0$, $s_{\DD}(C_3\cup C_3)>0$, or
\item $s_{\DD}(C_3\oplus C_3)=s_{\DD}(C_3\oplus C_5)=s_{\DD}(C_3\cup C_3)=s_{\DD}(C_3\cup C_5)=0$.
\end{itemize}
\end{corollary}

\section{Decks with ten edges}
\label{sec:deck10}
In this section we prove Theorem~\ref{thm:deck10}, which determines classes of $10$-decks;
the statement is illustrated in Figure~\ref{fig:deck10}.
Lemmas \ref{lm:deck10-III}, \ref{lm:deck10-I}, and \ref{lm:deck10-II} describe when a given $10$-deck is of 
Class III, Class I, or Class II, respectively.
\begin{figure}
\begin{center}
\epsfbox{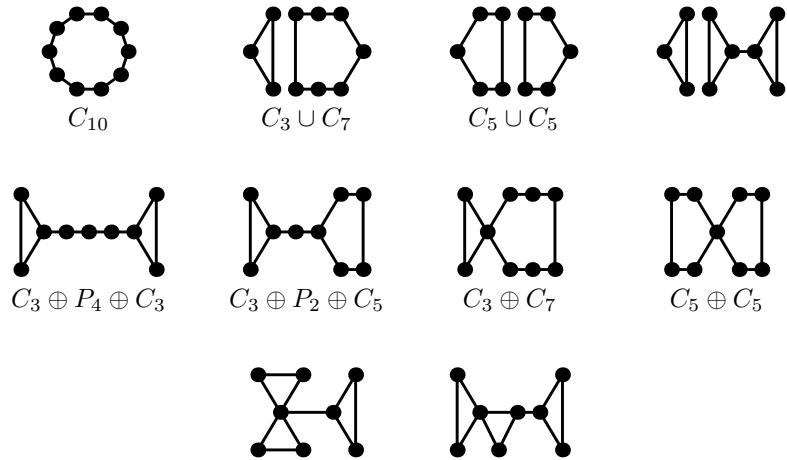}
\end{center}
\caption{Principal $10$-edge graphs.}
\label{fig:princ10}
\end{figure}
\begin{lemma}
\label{lm:deck10-III}
A $10$-deck $\DD$ is of Class III if the $8$-deck of $\DD$ is of Class III,
$s_{\DD}(C_{10})=0$ and at least one of the following holds:
\begin{itemize}
\item $s_{\DD}(C_3\oplus C_3)>0$,
\item $s_{\DD}(C_3\cup C_3)>0$, $s_{\DD}(C_3\oplus C_7)=0$ and $s_{\DD}(C_3\oplus P_2\oplus C_3)>0$,
\item $s_{\DD}(C_3\cup C_3)>0$, $s_{\DD}(C_3\oplus C_7)=0$ and $4s_{\DD}(C_3\oplus P_4\oplus C_3)s_{\DD}(C_5\oplus C_5)\ge s_{\DD}(C_3\oplus P_2\oplus C_5)^2$, or
\item $s_{\DD}(C_3\oplus C_3)=s_{\DD}(C_3\cup C_3)=s_{\DD}(C_3\oplus C_7)=s_{\DD}(C_3\cup C_7)=0$.
\end{itemize}
\end{lemma}

\begin{proof}
Fix a $10$-deck $\DD$ such that the $8$-deck of $\DD$ is of Class III, $s_{\DD}(C_{10})=0$ and
that satisfies at least one of the four cases given in the statement of the lemma.
By Corollary~\ref{cor:deck8-III}, it holds that $s_{\DD}(C_{\ell})=0$ for $\ell=2,4,6,8$.
Hence, Lemma~\ref{lm:core-III} yields that there exists a non-zero kernel $U$
such that $c^U_{\DD,2}=\cdots=c^U_{\DD,10}=0$.

We next fix a non-zero kernel $U$ and show that
either all the coefficients $c^U_{\DD,2},\ldots,c^U_{\DD,10}$ are zero or
at least one of them is non-zero and the first non-zero among them is positive.
If $U$ is not balanced, then $c^U_{\DD,2}> 0$ by Proposition~\ref{prop:coeff2}.
Otherwise, $t(H,U)=0$ for every graph $H$ with at most ten edges that is not principal.
This yields that $c^U_{\DD,2}=0$ and $c^U_{\DD,4}=0$.

We first assume that $s_{\DD}(C_3\oplus C_3)>0$.
As in the proof of Lemma~\ref{lm:deck8-III},
we observe that the coefficient $c^U_{\DD,6}$ is positive unless $t_U^{C_3}(x)=0$ for almost every $x\in [0,1]$.
In the latter case, $t(H,U)=0$ for every principal graph $H$ with at most ten edges
unless $H$ is an even cycle, $H=C_5\cup C_5$ or $H=C_5\oplus C_5$.
Since $t(C_5\cup C_5,U)\ge 0$ and $t(C_5\oplus C_5,U)\ge 0$,
we conclude that the $10$-deck $\DD$ is indeed of Class III.

We assume in the rest of the proof that $s_{\DD}(C_3\oplus C_3)=0$ (otherwise, the first case of the lemma applies);
Corollary~\ref{cor:deck8-III} implies that $s_{\DD}(C_3\oplus C_5)=0$.
Since all the three remaining cases also include the assumption that $s_{\DD}(C_3\oplus C_7)=0$,
we will also assume that $s_{\DD}(C_3\oplus C_7)=0$.
In addition, it holds that $s_{\DD}(C_3\oplus P_2\oplus C_3)>0$ in the second case that we consider.
The same arguments as presented in the proof of Lemma~\ref{lm:deck8-III} yields that
the coefficient $c^U_{\DD,6}=s_{\DD}(C_3\cup C_3)t(C_3,U)^2$ is non-negative and
it is equal to zero if and only if $t(C_3,U)=0$.
If $t(C_3,U)=0$, then $c^U_{\DD,8}=s_{\DD}(C_3\oplus P_2\oplus C_3)t(C_3\oplus P_2\oplus C_3,U)$,
which implies that $c^U_{\DD,8}$ is non-negative and
it is equal to zero if and only if $t^{C_3\oplus P_1}_U(x)=0$ for almost every $x\in [0,1]$.
Hence, $c^U_{\DD,6}=0$ and $c^U_{\DD,8}=0$ if and only if $t(C_3,U)=0$ and $t^{C_3\oplus P_1}_U(x)=0$ for almost every $x\in [0,1]$;
otherwise, the first non-zero of these two coefficients is positive.
If $t(C_3,U)=0$ and $t^{C_3\oplus P_1}_U(x)=0$ for almost every $x\in [0,1]$,
then the coefficient $c^U_{\DD,10}$ is equal to
\[c^U_{\DD,10}=s_{\DD}(C_5\cup C_5)t(C_5\cup C_5,U)+s_{\DD}(C_5\oplus C_5)t(C_5\oplus C_5,U),\]
i.e., $c^U_{\DD,10}$ is non-negative.
This concludes the analysis of the case when $s_{\DD}(C_3\oplus P_2\oplus C_3)>0$.

We assume in the rest of the proof that $s_{\DD}(C_3\oplus P_2\oplus C_3)=0$
in addition to $s_{\DD}(C_3\oplus C_3)=s_{\DD}(C_3\oplus C_5)=s_{\DD}(C_3\oplus C_7)=0$.
In the third case, it holds that $s_{\DD}(C_3\cup C_3)>0$ and 
$4s_{\DD}(C_3\oplus P_4\oplus C_3)s_{\DD}(C_5\oplus C_5)\ge s_{\DD}(C_3\oplus P_2\oplus C_5)^2$.
Since $c^U_{\DD,6}=s_{\DD}(C_3\cup C_3)t(C_3,U)^2$,
we conclude that either $c^U_{\DD,6}$ is positive or $t(C_3,U)=0$.
In the latter case, it holds that
\[c^U_{\DD,8}=s_{\DD}(C_3\oplus P_2\oplus C_3)t(C_3\oplus P_2\oplus C_3,U)\ge 0\]
and
\begin{align*}
 c^U_{\DD,10} =& s_{\DD}(C_3\oplus P_4\oplus C_3)t(C_3\oplus P_4\oplus C_3,U)+\\
               & s_{\DD}(C_3\oplus P_2\oplus C_5)t(C_3\oplus P_2\oplus C_5,U)+\\
	       & s_{\DD}(C_5\oplus C_5)t(C_5\oplus C_5,U)+s_{\DD}(C_5\cup C_5)t(C_5\cup C_5,U).
\end{align*}
Since it holds that $t(C_5\cup C_5,U)=t(C_5,U)^2\ge 0$, it follows that
\begin{align*}
 c^U_{\DD,10}\ge\int_{[0,1]}
                \begin{pmatrix}t_U^{C_3\oplus P_2}(x) \\ t_U^{C_5}(x) \end{pmatrix}^T
                \begin{pmatrix}
                s_{\DD}(C_3\oplus P_4\oplus C_3) & \frac{s_{\DD}(C_3\oplus P_2\oplus C_5)}{2} \\
                \frac{s_{\DD}(C_3\oplus P_2\oplus C_5)}{2} & s_{\DD}(C_5\oplus C_5)
                \end{pmatrix}
                \begin{pmatrix}t_U^{C_3\oplus P_2}(x) \\ t_U^{C_5}(x) \end{pmatrix}\dd x.
\end{align*}
The assumption that $4s_{\DD}(C_3\oplus P_4\oplus C_3)s_{\DD}(C_5\oplus C_5)\ge s_{\DD}(C_3\oplus P_2\oplus C_5)^2$
implies that the matrix
\[\begin{pmatrix}
  s_{\DD}(C_3\oplus P_4\oplus C_3) & \frac{s_{\DD}(C_3\oplus P_2\oplus C_5)}{2} \\
  \frac{s_{\DD}(C_3\oplus P_2\oplus C_5)}{2} & s_{\DD}(C_5\oplus C_5)
  \end{pmatrix}\]
is positive semidefinite,
which yields that the product in the integral above is non-negative for every $x\in [0,1]$.
It follows that $c^U_{\DD,10}\ge 0$.
This concludes the analysis of the third case of the lemma.

It remains to analyze the final case.
In this case, it holds that $s_{\DD}(C_3\cup C_3)=0$ and $s_{\DD}(C_3\oplus C_3)=0$,
which yields that $s_{\DD}(C_3\cup C_5)=0$ by Corollary~\ref{cor:deck8-III}.
It follows that $c^U_{\DD,2}=c^U_{\DD,4}=c^U_{\DD,6}=c^U_{\DD,8}=0$ and
$s_{\DD}(H)$ can be non-zero only for the following $10$-edge principal graphs $H$:
$C_5\cup C_5$, $C_5\oplus C_5$ and $C_3\oplus P_4\oplus C_3$.
Since $t(H,U)\ge 0$ for each of these graphs $H$,
we obtain that $c^U_{\DD,10}$ is non-negative.
Hence, the $10$-deck $\DD$ is of Class III.
\end{proof}

\begin{lemma}
\label{lm:deck10-I}
A $10$-deck $\DD$ is of Class I if either the $8$-deck of $\DD$ is of Class I or
the $8$-deck of $\DD$ is of Class III, $s_{\DD}(C_{10})>0$ and at least one of the following holds:
\begin{itemize}
\item $s_{\DD}(C_3\oplus C_3)>0$,
\item $s_{\DD}(C_3\cup C_3)>0$, $s_{\DD}(C_3\oplus C_7)=0$ and $s_{\DD}(C_3\oplus P_2\oplus C_3)>0$,
\item $s_{\DD}(C_3\cup C_3)>0$, $s_{\DD}(C_3\oplus C_7)=0$ and $4s_{\DD}(C_3\oplus P_4\oplus C_3)s_{\DD}(C_5\oplus C_5)\ge s_{\DD}(C_3\oplus P_2\oplus C_5)^2$, or
\item $s_{\DD}(C_3\oplus C_3)=s_{\DD}(C_3\cup C_3)=s_{\DD}(C_3\oplus C_7)=s_{\DD}(C_3\cup C_7)=0$.
\end{itemize}
\end{lemma}

\begin{proof}
If the $8$-deck of $\DD$ is of Class I, then the $10$-deck $\DD$ is also of Class I.
We next assume that the $8$-deck of $\DD$ is of Class III.
Let $\DD'$ be the $10$-deck obtained from $\DD$ by removing all cycles of length ten.
Since the $10$-deck $\DD'$ satisfies the assumptions of Lemma~\ref{lm:deck10-III},
the $10$-deck $\DD'$ is of Class III.
It follows that for any non-zero kernel $U$,
all the coefficients $c^U_{\DD',2},\ldots,c^U_{\DD',10}$ are zero or
the first non-zero among these coefficients is positive.
Since $c^U_{\DD,\ell}=c^U_{\DD',\ell}$ for $\ell=2,4,6,8$ and $c^U_{\DD,10}=c^U_{\DD',10}+s_{\DD}(C_{10})t(C_{10},U)$,
we obtain using Proposition~\ref{prop:cycle} that
at least one of the coefficients $c^U_{\DD,2},\ldots,c^U_{\DD,10}$ is non-zero and
the first non-zero among these coefficients is positive.
This implies that the $10$-deck $\DD$ is of Class I.
\end{proof}

\begin{lemma}
\label{lm:deck10-II}
A $10$-deck $\DD$ is of Class II if either the $8$-deck of $\DD$ is of Class II or
the $8$-deck of $\DD$ is of Class III and at least one of the following holds:
\begin{itemize}
\item $s_{\DD}(C_3\oplus C_3)=0$ and $s_{\DD}(C_3\oplus C_7)>0$,
\item $s_{\DD}(C_3\oplus C_3)=s_{\DD}(C_3\cup C_3)=0$ and $s_{\DD}(C_3\cup C_7)>0$, or
\item $s_{\DD}(C_3\oplus C_3)=s_{\DD}(C_3\oplus P_2\oplus C_3)=0$ and $4s_{\DD}(C_3\oplus P_4\oplus C_3)s_{\DD}(C_5\oplus C_5)<s_{\DD}(C_3\oplus P_2\oplus C_5)^2$.
\end{itemize}
\end{lemma}

\begin{proof}
Fix a $10$-deck $\DD$.
If the $8$-deck of $\DD$ is of Class I or II, then there is nothing to prove.
Hence, we assume that the $8$-deck of $\DD$ is of Class III and
analyze each of the three cases listed in the statement of the lemma separately.
Note that $s_{\DD}(C_\ell)=0$ for $\ell=2,4,6,8$.

In the first case, we apply Lemma~\ref{lm:core} with $k=9$, $m=1$, any $\delta\in (0,1)$ such that $s_{\DD}(C_{10})\delta\le 1/2$,
$\gamma_3=\gamma_5=\gamma_7=\gamma_9=0$, $\sigma_1=0$, $\tau_{1,3}=1$, $\tau_{1,5}=\tau_{1,9}=0$ and $\tau_{1,7}=-1$
to get a balanced kernel $U$ with the properties given in the statement of Lemma~\ref{lm:core}.
Note that $t(H,U)=0$ for all principal graphs with at most $10$ edges
with the exception of $H$ being an even cycle, $C_3\oplus C_3$ or $C_3\oplus C_7$.
It follows that $c^U_{\DD,\ell}=0$ for $\ell=2,4,6,8$ and
\[c^U_{\DD,10}=s_{\DD}(C_{10})t(C_{10},U)+s_{\DD}(C_3\oplus C_7)t(C_3\oplus C_7,U)\le -1/2.\]
Hence, the $10$-deck $\DD$ is of Class II.

In the second case, we apply Lemma~\ref{lm:core} with $k=9$, $m=0$, any $\delta\in (0,1)$ such that $s_{\DD}(C_{10})\delta\le 1/2$,
$\gamma_3=1$, $\gamma_5=\gamma_9=0$ and $\gamma_7=-1$ to get a balanced kernel $U$.
Note that $t(H,U)=0$ for all principal graphs with at most $10$ edges
with the exception of $H$ being an even cycle, $C_3\cup C_3$, $C_3\oplus C_3$ or $C_3\cup C_7$.
It follows that $c^U_{\DD,\ell}=0$ for $\ell=2,4,6,8$ and
\[c^U_{\DD,10}=s_{\DD}(C_{10})t(C_{10},U)+s_{\DD}(C_3\cup C_7)t(C_3\cup C_7,U)\le -1/2.\]
Hence, the $10$-deck $\DD$ is of Class II in this case, too.

It remains to consider the final case given in the statement of the lemma.
Since $4s_{\DD}(C_3\oplus P_4\oplus C_3)s_{\DD}(C_5\oplus C_5)-s_{\DD}(C_3\oplus P_2\oplus C_5)^2$ is negative,
the matrix
\[\begin{pmatrix}
  s_{\DD}(C_3\oplus P_4\oplus C_3) & \frac{s_{\DD}(C_3\oplus P_2\oplus C_5)}{2} \\
  \frac{s_{\DD}(C_3\oplus P_2\oplus C_5)}{2} & s_{\DD}(C_5\oplus C_5)
  \end{pmatrix}\]
has a negative eigenvalue, i.e., there exists a vector $(z_3,z_5)\in\RR^2$ such that
\[\begin{pmatrix} z_3 \\ z_5 \end{pmatrix}^T
  \begin{pmatrix}
  s_{\DD}(C_3\oplus P_4\oplus C_3) & \frac{s_{\DD}(C_3\oplus P_2\oplus C_5)}{2} \\
  \frac{s_{\DD}(C_3\oplus P_2\oplus C_5)}{2} & s_{\DD}(C_5\oplus C_5)
  \end{pmatrix}
  \begin{pmatrix} z_3 \\ z_5 \end{pmatrix}=-1.\]
We next apply Lemma~\ref{lm:core} with $k=9$, $m=1$, any $\delta\in (0,1)$ such that $s_{\DD}(C_{10})\delta\le 1/2$,
$\gamma_3=\gamma_5=\gamma_7=\gamma_9=0$, $\sigma_1=\delta/2$, $\tau_{1,3}=4z_3/\delta^2$, $\tau_{1,5}=z_5$ and $\tau_{1,7}=\tau_{1,9}=0$
to get a balanced kernel $U$ with the properties given in the statement of Lemma~\ref{lm:core};
let $f_1$ be the eigenfunction from the statement of Lemma~\ref{lm:core}.
Note that $t(C_3,U)=t(C_5,U)=t(C_7,U)=0$.
Since $s_{\DD}(C_3\oplus C_3)=0$, it holds that $s_{\DD}(C_3\oplus C_5)=0$ by Corollary~\ref{cor:deck8-III}.
It follows that $c^U_{\DD,\ell}=0$ for $\ell=2,4,6,8$.
For every $10$-edge principal graph $H$, it holds that $s_{\DD}(H)=0$ or $t(H,U)=0$
unless $H$ is $C_3\oplus P_4\oplus C_3$, $C_3\oplus P_2\oplus C_5$ or $C_5\oplus C_5$ (here,
we use that $s_{\DD}(H)=0$ for every $h$ containing $C_3\oplus C_3$ as $s_{\DD}(C_3\oplus C_3)=0$).
It follows that
\begin{align*}
 c^U_{\DD,10} =& s_{\DD}(C_{10})t(C_{10},U)+s_{\DD}(C_3\oplus P_4\oplus C_3)t(C_3\oplus P_4\oplus C_3,U)+\\
               & s_{\DD}(C_3\oplus P_2\oplus C_5)t(C_3\oplus P_2\oplus C_5,U)+s_{\DD}(C_5\oplus C_5)t(C_5\oplus C_5,U)
\end{align*}
which can be rewritten as
\begin{align*}
 c^U_{\DD,10} =& s_{\DD}(C_{10})t(C_{10},U)+\\
              & \int_{[0,1]}
                \begin{pmatrix}t_U^{C_3\oplus P_2}(x) \\ t_U^{C_5}(x) \end{pmatrix}^T
                \begin{pmatrix}
                s_{\DD}(C_3\oplus P_4\oplus C_3) & \frac{s_{\DD}(C_3\oplus P_2\oplus C_5)}{2} \\
                \frac{s_{\DD}(C_3\oplus P_2\oplus C_5)}{2} & s_{\DD}(C_5\oplus C_5)
                \end{pmatrix}
                \begin{pmatrix}t_U^{C_3\oplus P_2}(x) \\ t_U^{C_5}(x) \end{pmatrix}\dd x.
\end{align*}
Since $t_U^{C_3\oplus P_2}(x)=\sigma_1^2 \tau_{1,3} f_1(x)=z_3 f_1(x)$ 
and $t_U^{C_5}(x)=\tau_{1,5} f_1(x) = z_5 f_1(x)$,
we obtain that
\begin{align*}
 c^U_{\DD,10} =& s_{\DD}(C_{10})t(C_{10},U)+\\
              & \int_{[0,1]}
                \begin{pmatrix} z_3 \\ z_5 \end{pmatrix}^T
                \begin{pmatrix}
                s_{\DD}(C_3\oplus P_4\oplus C_3) & \frac{s_{\DD}(C_3\oplus P_2\oplus C_5)}{2} \\
                \frac{s_{\DD}(C_3\oplus P_2\oplus C_5)}{2} & s_{\DD}(C_5\oplus C_5)
                \end{pmatrix}
		\begin{pmatrix} z_3 \\ z_5 \end{pmatrix} f_1(x)^2 \dd x.
\end{align*}
Since the integral is equal to $-1$ as the $L_2$-norm of $f_1$ is one,
it follows that $c^U_{\DD,10}\le -1/2$.
We conclude that the $10$-deck $\DD$ is of Class II.
\end{proof}

\begin{figure}
\begin{center}
\epsfbox{dcommon-13.mps}
\end{center}
\caption{The classification of $10$-decks $\DD$ with $s_{\DD}(P_2)>0$ whose $8$-decks is of Class III;
         we omit the subscript $\DD$ in the diagram.}
\label{fig:deck10}
\end{figure}

We are now ready to state the main theorem of this section.

\begin{theorem}
\label{thm:deck10}
Let $\DD$ be a $10$-deck with $s_{\DD}(P_2)>0$.
If the $8$-deck of $\DD$ is of Class I or of Class II, then $\DD$ is of Class I or of Class II, respectively.
Otherwise, the deck $\DD$ is of Class I, Class II or Class III as determined
in the diagram in Figure~\ref{fig:deck10}.
\end{theorem}

\begin{proof}
The proof follows by inspecting the diagram in Figure~\ref{fig:deck10} and verifying that
every path leading to the label Class I corresponds to the assumptions of Lemma~\ref{lm:deck10-I},
every path leading to the label Class II corresponds to the assumptions of Lemma~\ref{lm:deck10-II}, and
every path leading to the label Class III corresponds to the assumptions of Lemma~\ref{lm:deck10-III}.
\end{proof}

Theorem~\ref{thm:deck10} yields the following corollary.

\begin{corollary}
\label{cor:deck10-III-classify}
Every $10$-deck $\DD$ of Class III satisfies that $s_{\DD}(P_2)>0$, $s_{\DD}(C_4)=s_{\DD}(C_6)=s_{\DD}(C_8)=s_{\DD}(C_{10})=0$ and
exactly one of the following:
\begin{itemize}
\item $s_{\DD}(C_3\oplus C_3)>0$,
\item $s_{\DD}(C_3\oplus C_3)=s_{\DD}(C_3\oplus C_5)=s_{\DD}(C_3\oplus C_7)=s_{\DD}(C_3\cup C_3)=s_{\DD}(C_3\cup C_5)=s_{\DD}(C_3\cup C_7)=0$,
\item $s_{\DD}(C_3\oplus C_3)=s_{\DD}(C_3\oplus C_5)=s_{\DD}(C_3\oplus C_7)=0$, $s_{\DD}(C_3\cup C_3)>0$ and $s_{\DD}(C_3\oplus P_2\oplus C_3)>0$, or
\item $s_{\DD}(C_3\oplus C_3)=s_{\DD}(C_3\oplus C_5)=s_{\DD}(C_3\oplus C_7)=s_{\DD}(C_3\oplus P_2\oplus C_3)=0$, $s_{\DD}(C_3\cup C_3)>0$ and $4s_{\DD}(C_3\oplus P_4\oplus C_3)s_{\DD}(C_5\oplus C_5)\ge s_{\DD}(C_3\oplus P_2\oplus C_5)^2$.
\end{itemize}
\end{corollary}

\section{Decks with twelve edges}
\label{sec:deck12}
In this section,
we analyze $12$-decks such that their $10$-decks are of Class III (Lemmas~\ref{lm:deck12-sC3C3}--\ref{lm:deck12-det}) and
prove our main result, Theorem~\ref{thm:deck12};
the statement of the theorem is illustrated in Figure~\ref{fig:deck12}.
The first three lemmas cover the first three cases described in Corollary~\ref{cor:deck10-III-classify} respectively.

\begin{figure}
\begin{center}
\epsfbox{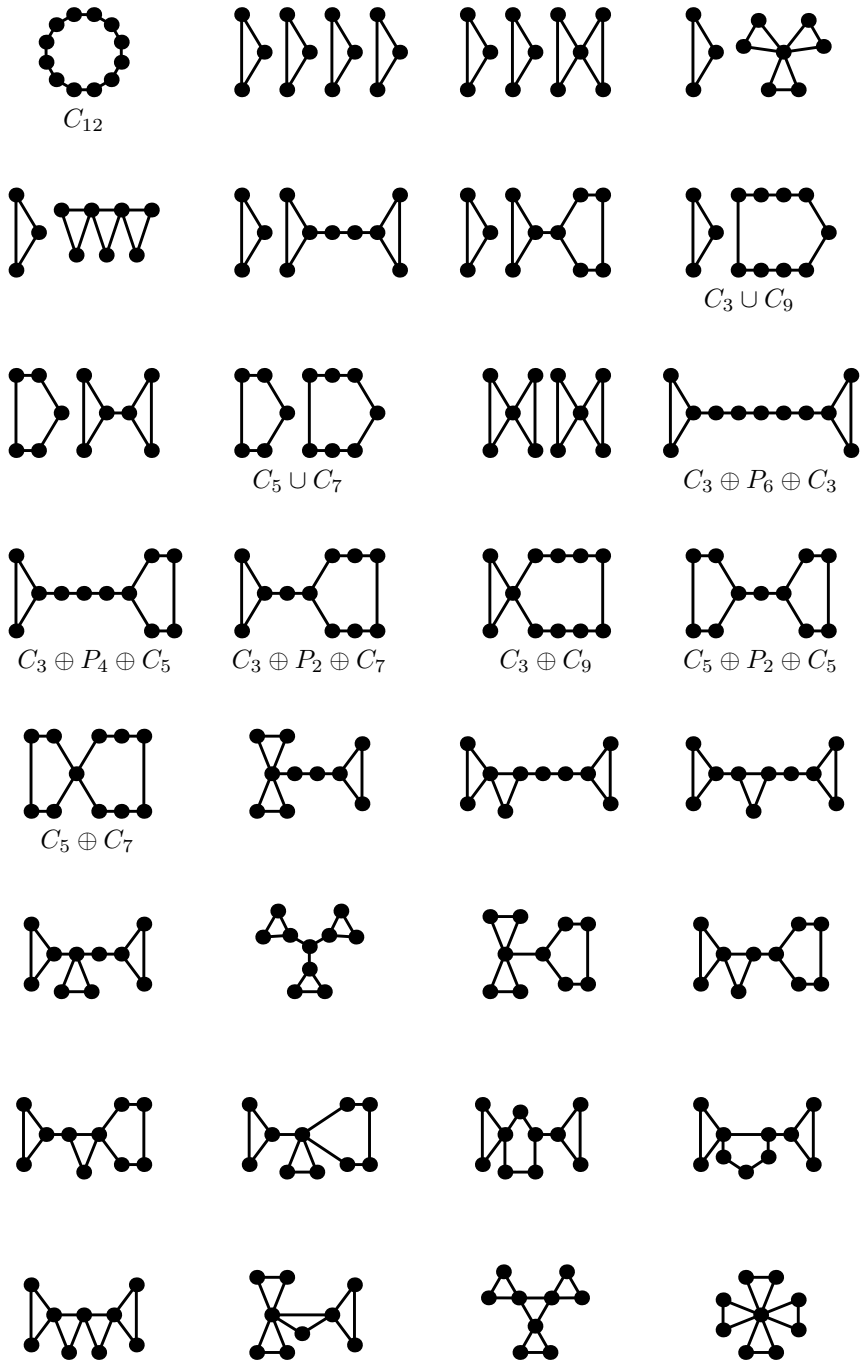}
\end{center}
\caption{Principal $12$-edge graphs.}
\label{fig:princ12}
\end{figure}

\begin{lemma}
\label{lm:deck12-sC3C3}
Let $\DD$ be a $12$-deck such that its $10$-deck is of Class III, $s_{\DD}(P_2)>0$ and $s_{\DD}(C_3\oplus C_3)>0$.
If
\begin{itemize}
\item $s_{\DD}(C_5\oplus C_5)=0$ and $s_{\DD}(C_5\oplus C_7)>0$, or
\item $s_{\DD}(C_5\oplus C_5)=s_{\DD}(C_5\cup C_5)=0$ and $s_{\DD}(C_5\cup C_7)>0$,
\end{itemize}
then $\DD$ is of Class II.
Otherwise, $\DD$ is of Class I if $s_{\DD}(C_{12})>0$ and of Class III if $s_{\DD}(C_{12})=0$.
\end{lemma}

\begin{proof}
Let $\DD$ be a $12$-deck $\DD$ such that its $10$-deck is of Class III, $s_{\DD}(P_2)>0$ and $s_{\DD}(C_3\oplus C_3)>0$.
Note that $s_{\DD}(C_{\ell})=0$ for $\ell=4,6,8,10$ by Corollary~\ref{cor:deck10-III-classify}.
We first show that if $\DD$ satisfies one of the two conditions in the statement of the lemma, then $\DD$ is of Class II.
The first case to consider is when $s_{\DD}(C_5\oplus C_5)=0$ and $s_{\DD}(C_5\oplus C_7)>0$.
We apply Lemma~\ref{lm:core} with $k=11$, $m=1$, any $\delta\in (0,1)$ such that $s_{\DD}(C_{12})\delta\le 1/2$,
$\gamma_3=\gamma_5=\gamma_7=\gamma_9=\gamma_{11}=0$, $\sigma_1=0$, $\tau_{1,3}=\tau_{1,9}=\tau_{1,11}=0$, $\tau_{1,5}=1$ and $\tau_{1,7}=-1$
to get a non-zero kernel $U$ with the properties given in Lemmas~\ref{lm:core} and~\ref{lm:core-apply}.
It holds that $t(H,U)=0$ for all principal subgraphs with at most twelve edges with the exception 
of $H$ being an even cycle, $C_5\oplus C_5$ or $C_5\oplus C_7$.
It follows that $c^U_{\DD,2}=\cdots=c^U_{\DD,10}=0$ and
\[c^U_{\DD,12}=s_{\DD}(C_{12})t(C_{12},U)+s_{\DD}(C_5\oplus C_7)t(C_5\oplus C_7,U)\le -1/2.\]
Hence, the deck $\DD$ is of Class II.

The second case is when $s_{\DD}(C_5\oplus C_5)=s_{\DD}(C_5\cup C_5)=0$ and $s_{\DD}(C_5\cup C_7)>0$.
We apply Lemma~\ref{lm:core} with $k=11$, $m=0$, any $\delta\in (0,1)$ such that $s_{\DD}(C_{12})\delta\le 1/2$,
$\gamma_3=\gamma_9=\gamma_{11}=0$, $\gamma_5=1$ and $\gamma_7=-1$
to get a non-zero kernel $U$ with the properties given in Lemmas~\ref{lm:core} and~\ref{lm:core-apply}.
It holds that $t(H,U)=0$ for all principal subgraphs with at most twelve edges with the exception
of $H$ being an even cycle, $C_5\oplus C_5$, $C_5\cup C_5$ or $C_5\cup C_7$.
Hence, it holds that $c^U_{\DD,2}=\cdots=c^U_{\DD,10}=0$ and
\[c^U_{\DD,12}=s_{\DD}(C_{12})t(C_{12},U)+s_{\DD}(C_5\cup C_7)t(C_5\cup C_7,U)\le -1/2,\]
which yields that the deck $\DD$ is of Class II.

We now prove that if the deck $\DD$ does not satisfy any of the two conditions in the statement of the lemma,
then it is of Class I or Class III.
Following the line of arguments presented in the proofs of Lemmas~\ref{lm:deck8-I} and~\ref{lm:deck10-I},
it is enough to establish that the deck $\DD$ is of Class III when $s_{\DD}(C_{12})=0$;
note that if $s_{\DD}(C_{12})=0$, then the deck $\DD$ is not of Class I by Lemma~\ref{lm:core-III}.

Let $U$ be an arbitrary non-zero kernel.
The coefficient $c^U_{\DD,2}$ is positive unless $U$ is balanced.
Hence, we can assume that $U$ is balanced.
Since $s_{\DD}(C_4)=0$, we obtain that $c^U_{\DD,4}=0$.
Since all $t(C_6,U)$, $t(C_3\oplus C_3,U)$ and $t(C_3\cup C_3,U)$ are non-negative,
we obtain that $c^U_{\DD,6}\ge s_{\DD}(C_3\oplus C_3)t(C_3\oplus C_3,U)$.
It follows that $c^U_{\DD,6}$ is positive unless $t^{C_3}_U(x)=0$ for almost every $x\in [0,1]$.
Hence, we can further assume that $t^{C_3}_U(x)=0$ for almost every $x\in [0,1]$.
This implies that $t(H,U)=0$ for all principal graphs with eight or ten edges
with the exception of $H$ being $C_8$, $C_{10}$, $C_5\cup C_5$ or $C_5\oplus C_5$.
Hence, the coefficient $c^U_{\DD,8}$ is zero and the coefficient $c^U_{\DD,10}$ is non-negative.

If $s_{\DD}(C_5\oplus C_5)>0$, then $c^U_{\DD,10}$ is positive unless $t^{C_5}_U(x)=0$ for almost every $x\in [0,1]$;
in the latter case, $t(H,U)=0$ for every principal $12$-edge graph $H$ with the exception of $C_{12}$,
which yields that $c^U_{\DD,12}=0$.
If $s_{\DD}(C_5\oplus C_5)=0$ and $s_{\DD}(C_5\cup C_5)>0$, then $c^U_{\DD,10}$ is positive unless $t(C_5,U)=0$;
if $t(C_5,U)=0$, then $s_{\DD}(H)=0$ or $t(H,U)=0$ for every principal $12$-edge graph $H$ with the exception of $C_{12}$, $C_5\oplus P_2\oplus C_5$ and $C_5\oplus C_7$.
In particular, unless the first case described in the statement of the lemma applies, the coefficient $c^U_{\DD,12}$ is non-negative.
Finally, if $s_{\DD}(C_5\oplus C_5)=s_{\DD}(C_5\cup C_5)=0$,
then $s_{\DD}(H)=0$ or $t(H,U)=0$ for every principal $12$-edge graph $H$ with the exception of $C_{12}$, $C_5\oplus P_2\oplus C_5$, $C_5\oplus C_7$ and $C_5\cup C_7$, and $c^U_{\DD,12}$ is non-negative unless the second case in the statement of the lemma applies.
We conclude that $\DD$ is of Class III in either of the three cases distinguished in this paragraph.
\end{proof}

\begin{lemma}
\label{lm:deck12-sC3C3x}
Let $\DD$ be a $12$-deck such that its $10$-deck is of Class III, $s_{\DD}(P_2)>0$, $s_{\DD}(C_3\oplus C_3)=s_{\DD}(C_3\cup C_3)=0$.
If
\begin{itemize}
\item $s_{\DD}(C_3\oplus C_9)>0$,
\item $s_{\DD}(C_3\cup C_9)>0$,
\item $s_{\DD}(C_5\oplus C_5)=0$ and $s_{\DD}(C_5\oplus C_7)>0$, or
\item $s_{\DD}(C_5\oplus C_5)=s_{\DD}(C_5\cup C_5)=0$ and $s_{\DD}(C_5\cup C_7)>0$,
\end{itemize}
then $\DD$ is of Class II.
Otherwise, $\DD$ is of Class I if $s_{\DD}(C_{12})>0$ and of Class III if $s_{\DD}(C_{12})=0$.
\end{lemma}

\begin{proof}
Let $\DD$ be a $12$-deck $\DD$ such that its $10$-deck is of Class III, $s_{\DD}(P_2)>0$ and $s_{\DD}(C_3\oplus C_3)=s_{\DD}(C_3\cup C_3)=0$.
Note that $s_{\DD}(C_{\ell})=0$ for $\ell=4,6,8,10$ and
$s_{\DD}(C_3\oplus C_{\ell})=s_{\DD}(C_3\cup C_{\ell})=0$ for $\ell=5,7$ by Corollary~\ref{cor:deck10-III-classify}.
Note that $s_{\DD}(H)$ can be positive only for the following principal graphs $H$ with at most twelve edges:
$C_5\cup C_5$, $C_5\oplus C_5$, $C_{12}$, $C_3\cup C_9$, $C_3\oplus C_9$, $C_5\oplus P_2\oplus C_5$, $C_5\cup C_7$ and $C_5\oplus C_7$.

We first show that if $\DD$ satisfies one of the four conditions in the statement of the lemma, then $\DD$ is of Class II.
In the first two cases, we apply Lemma~\ref{lm:core} with $k=11$, $m=1$, any $\delta\in (0,1)$ such that $s_{\DD}(C_{12})\delta\le 1/2$,
$\gamma_3=1$, $\gamma_9=-1$, $\gamma_5=\gamma_7=\gamma_{11}=0$, $\sigma_1=0$, $\tau_{1,3}=1$, $\tau_{1,9}=-1$ and $\tau_{1,5}=\tau_{1,7}=\tau_{1,11}=0$, and
we get a non-zero kernel $U$ that satisfies the properties listed in Lemma~\ref{lm:core}.
Since $t(C_5\cup C_5,U)$, $t(C_5\oplus C_5,U)$, $t(C_5\oplus P_2\oplus C_5,U)$, $t(C_5\cup C_7,U)$ and $t(C_5\oplus C_7,U)$ are equal to zero,
it follows that $c^U_{\DD,2}=\cdots=c^U_{\DD,10}=0$ and
\begin{align*}
c^U_{\DD,12} &=s_{\DD}(C_{12})t(C_{12},U)+s_{\DD}(C_3\oplus C_9)t(C_3\oplus C_9,U)\\
             &+s_{\DD}(C_3\cup C_9)t(C_3\cup C_9,U)\\
             &\le s_{\DD}(C_{12})t(C_{12},U)-1 \le -1/2.
\end{align*}	       
Hence, the deck $\DD$ is of Class II.

We next consider the case that $s_{\DD}(C_5\oplus C_5)=0$ and $s_{\DD}(C_5\oplus C_7)>0$.
In this case, we apply Lemma~\ref{lm:core} with $k=11$, $m=1$, any $\delta\in (0,1)$ such that $s_{\DD}(C_{12})\delta\le 1/2$,
$\gamma_3=\gamma_5=\gamma_7=\gamma_9=\gamma_{11}=0$, $\sigma_1=0$, $\tau_{1,5}=1$, $\tau_{1,7}=-1$ and $\tau_{1,3}=\tau_{1,9}=\tau_{1,11}=0$
to get a non-zero kernel $U$.
Similarly to the previous case, it holds that $c^U_{\DD,2}=\cdots=c^U_{\DD,10}=0$ and
\[c^U_{\DD,12}=s_{\DD}(C_{12})t(C_{12},U)+s_{\DD}(C_5\oplus C_7)t(C_5\oplus C_7,U)\le -1/2.\]
In the final case when $s_{\DD}(C_5\oplus C_5)=s_{\DD}(C_5\cup C_5)=0$ and $s_{\DD}(C_5\cup C_7)>0$,
we apply Lemma~\ref{lm:core} with $k=11$, $m=0$, any $\delta\in (0,1)$ such that $s_{\DD}(C_{12})\delta\le 1/2$,
$\gamma_5=1$, $\gamma_7=-1$ and $\gamma_3=\gamma_9=\gamma_{11}=0$.
We obtain a non-zero kernel $U$ such that $c^U_{\DD,2}=\cdots=c^U_{\DD,10}=0$ and
\[c^U_{\DD,12}=s_{\DD}(C_{12})t(C_{12},U)+s_{\DD}(C_5\cup C_7)t(C_5\cup C_7,U)\le -1/2.\]
In both cases, we conclude that the deck $\DD$ is of Class II.

We now prove that if the deck $\DD$ does not satisfy any of the four conditions in the statement of the lemma,
then it is of Class I or Class III.
Following the line of arguments presented in the proofs of Lemmas~\ref{lm:deck8-I} and~\ref{lm:deck10-I},
it is enough to establish that the deck $\DD$ is of Class III when $s_{\DD}(C_{12})=0$;
note that if $s_{\DD}(C_{12})=0$, then the deck $\DD$ is not of Class I by Lemma~\ref{lm:core-III}.

Let $U$ be an arbitrary non-zero kernel.
The coefficient $c^U_{\DD,2}$ is positive unless $U$ is balanced.
Hence, we can assume that $U$ is balanced,
which implies that $c^U_{\DD,2}=\cdots=c^U_{\DD,8}=0$, 
since we already deduced that $s_{\DD}(H)=0$ for every principal graph $H$ with at most eight edges.
In addition, $c^U_{\DD,10}\ge 0$ and the equality holds only in the following three cases:
both $s_{\DD}(C_5\cup C_5)$ and $s_{\DD}(C_5\oplus C_5)$ are zero, or
$s_{\DD}(C_5\oplus C_5)$ is zero and $t(C_5,U)=0$, or
$t_U^{C_5}(x)=0$ for almost every $x\in [0,1]$.
It is now straightforward to verify that
if $c^U_{\DD,10}=0$ and none of the cases given in the statement of the lemma applies, then
\begin{align*}
c^U_{\DD,12} &=s_{\DD}(C_5\oplus P_2\oplus C_5)t(C_5\oplus P_2\oplus C_5)\\
             &=\int_{[0,1]}t^{C_5\oplus P_1}_U(x)^2\dd x\ge 0.
\end{align*}
We can now conclude that the deck $\DD$ is of Class III.
\end{proof}

\begin{lemma}
\label{lm:deck12-sC3P2C3}
Let $\DD$ be a $12$-deck such that its $10$-deck is of Class III, $s_{\DD}(P_2)>0$, $s_{\DD}(C_3\oplus C_3)=0$, $s_{\DD}(C_3\cup C_3)>0$ and $s_{\DD}(C_3\oplus P_2\oplus C_3)>0$.
If
\begin{itemize}
\item $s_{\DD}(C_3\oplus C_9)>0$,
\item $s_{\DD}(C_5\oplus C_5)=0$ and $s_{\DD}(C_5\oplus C_7)>0$, or
\item $s_{\DD}(C_5\oplus C_5)=s_{\DD}(C_5\cup C_5)=0$ and $s_{\DD}(C_5\cup C_7)>0$,
\end{itemize}
then $\DD$ is of Class II.
Otherwise, $\DD$ is of Class I if $s_{\DD}(C_{12})>0$ and of Class III if $s_{\DD}(C_{12})=0$.
\end{lemma}

\begin{proof}
Let $\DD$ be a $12$-deck $\DD$ such that its $10$-deck is of Class III,
$s_{\DD}(P_2)>0$, $s_{\DD}(C_3\oplus C_3)=0$, $s_{\DD}(C_3\cup C_3)>0$ and $s_{\DD}(C_3\oplus P_2\oplus C_3)>0$.
Note that $s_{\DD}(C_{\ell})=0$ for $\ell=4,6,8,10$ and
$s_{\DD}(C_3\oplus C_{\ell})=0$ for $\ell=5,7$ by Corollary~\ref{cor:deck10-III-classify}.

We first show that if $\DD$ satisfies one of the three conditions in the statement of the lemma, then $\DD$ is of Class II.
In the first case, we apply Lemma~\ref{lm:core} with $k=11$, $m=1$, any $\delta\in (0,1)$ such that $s_{\DD}(C_{12})\delta\le 1/2$,
$\gamma_3=\gamma_5=\gamma_7=\gamma_9=\gamma_{11}=0$, $\sigma_1=0$, $\tau_{1,3}=1$, $\tau_{1,9}=-1$ and $\tau_{1,5}=\tau_{1,7}=\tau_{1,11}=0$
to get a non-zero kernel $U$ that satisfies the properties listed in Lemma~\ref{lm:core}.
Observe that $s_{\DD}(H)=0$ or $t(H,U)=0$ for a principal graph $H$ with at most twelve edges
unless $H$ is an even cycle or $H=C_3\oplus C_9$.
In the second case, we apply Lemma~\ref{lm:core} with $k=11$, $m=1$, any $\delta\in (0,1)$ such that $s_{\DD}(C_{12})\delta\le 1/2$,
$\gamma_3=\gamma_5=\gamma_7=\gamma_9=\gamma_{11}=0$, $\sigma_1=0$, $\tau_{1,5}=1$, $\tau_{1,7}=-1$ and $\tau_{1,3}=\tau_{1,9}=\tau_{1,11}=0$ and
get a non-zero kernel $U$ such that
that $s_{\DD}(H)=0$ or $t(H,U)=0$ for a principal graph $H$ with at most twelve edges
unless $H$ is an even cycle or $H=C_5\oplus C_7$.
Finally, in the third case, we apply Lemma~\ref{lm:core} with $k=11$, $m=0$, any $\delta\in (0,1)$ such that $s_{\DD}(C_{12})\delta\le 1/2$,
$\gamma_5=1$, $\gamma_7=-1$ and $\gamma_3=\gamma_9=\gamma_{11}=0$ and
obtain a non-zero kernel $U$ such that
that $s_{\DD}(H)=0$ or $t(H,U)=0$ for a principal graph $H$ with at most twelve edges
unless $H$ is an even cycle or $H=C_5\cup C_7$.
In each of the three cases, it holds that $c^U_{\DD,2}=\cdots=c^U_{\DD,10}=0$ and $c^U_{\DD,12}\le -1/2$,
i.e., the deck $\DD$ is of Class II.

We now prove that if the deck $\DD$ does not satisfy any of the three conditions in the statement of the lemma,
then it is of Class I or Class III.
Following the line of arguments presented in the proofs of Lemmas~\ref{lm:deck8-I} and~\ref{lm:deck10-I},
it is enough to establish that the deck $\DD$ is of Class III when $s_{\DD}(C_{12})=0$;
note that if $s_{\DD}(C_{12})=0$, then the deck $\DD$ is not of Class I by Lemma~\ref{lm:core-III}.

Let $U$ be an arbitrary non-zero kernel.
The coefficient $c^U_{\DD,2}$ is positive unless $U$ is balanced.
Hence, we can assume that $U$ is balanced.
This implies that $c^U_{\DD,4}=0$, $c^U_{\DD,6}\ge 0$ and the inequality is strict unless $t(C_3,U)=0$.
If $t(C_3,U)=0$, then $c^U_{\DD,8}\ge 0$ and the inequality is strict unless $t^{C_3\oplus P_1}_U(x)=0$ for almost every $x\in [0,1]$.
We next assume that $c^U_{\DD,6}=0$ and $c^U_{\DD,8}=0$ and
observe that $s_{\DD}(H)=0$ or $t(H,U)=0$ for every principal graph $H$ possibly with the following exceptions:
$C_5\cup C_5$, $C_5\oplus C_5$, $C_5\oplus P_2\oplus C_5$, $C_5\cup C_7$ and $C_5\oplus C_7$ (note that $s_{\DD}(C_3\oplus C_9)=0$;
otherwise, the first case in the statement of the lemma applies).
It follows that $c^U_{\DD,10}\ge 0$ and the equality holds only in the following three cases:
both $s_{\DD}(C_5\cup C_5)$ and $s_{\DD}(C_5\oplus C_5)$ are zero, or
$s_{\DD}(C_5\oplus C_5)$ is zero and $t(C_5,U)=0$, or
$t_U^{C_5}(x)=0$ for almost every $x\in [0,1]$.
It is now straightforward to verify that
if $c^U_{\DD,10}=0$ and none of the last two cases given in the statement of the lemma applies,
then $c^U_{\DD,12}=s_{\DD}(C_5\oplus P_2\oplus C_5)t(C_5\oplus P_2\oplus C_5)\ge 0$.
We conclude that the deck $\DD$ is indeed of Class III.
\end{proof}

The final five lemmas concern the last case described in Corollary~\ref{cor:deck10-III-classify};
each deals with one of the four cases based on which of the two quantities $s_{\DD}(C_3\oplus P_4\oplus C_3)$ and $s_{\DD}(C_5\oplus C_5)$
are zero or positive;
the final two lemmas deal with the case when both quantities are positive.

\begin{lemma}
\label{lm:deck12-det00}
Let $\DD$ be a $12$-deck such that its $10$-deck is of Class III, $s_{\DD}(P_2)>0$, $s_{\DD}(C_3\oplus C_3)=s_{\DD}(C_3\oplus P_2\oplus C_3)=0$ and $s_{\DD}(C_3\cup C_3)>0$.
Further suppose that $s_{\DD}(C_3\oplus P_4\oplus C_3)=s_{\DD}(C_5\oplus C_5)=s_{\DD}(C_3\oplus P_2\oplus C_5)=0$.
If
\begin{itemize}
\item $s_{\DD}(C_3\oplus C_9)>0$,
\item $s_{\DD}(C_3\oplus P_2\oplus C_7)>0$,
\item $4s_{\DD}(C_3\oplus P_6\oplus C_3)s_{\DD}(C_5\oplus P_2\oplus C_5)<s_{\DD}(C_3\oplus P_4\oplus C_5)^2$,
\item $s_{\DD}(C_5\oplus C_7)>0$,
\item $s_{\DD}(C_5\cup C_5)=0$ and $s_{\DD}(C_5\cup C_3\oplus P_1\oplus C_3)>0$, or
\item $s_{\DD}(C_5\cup C_5)=0$ and $s_{\DD}(C_5\cup C_7)>0$,
\end{itemize}
then $\DD$ is of Class II.
Otherwise, $\DD$ is of Class I if $s_{\DD}(C_{12})>0$ and of Class III if $s_{\DD}(C_{12})=0$.
\end{lemma}

\begin{figure}
\begin{center}
\epsfbox{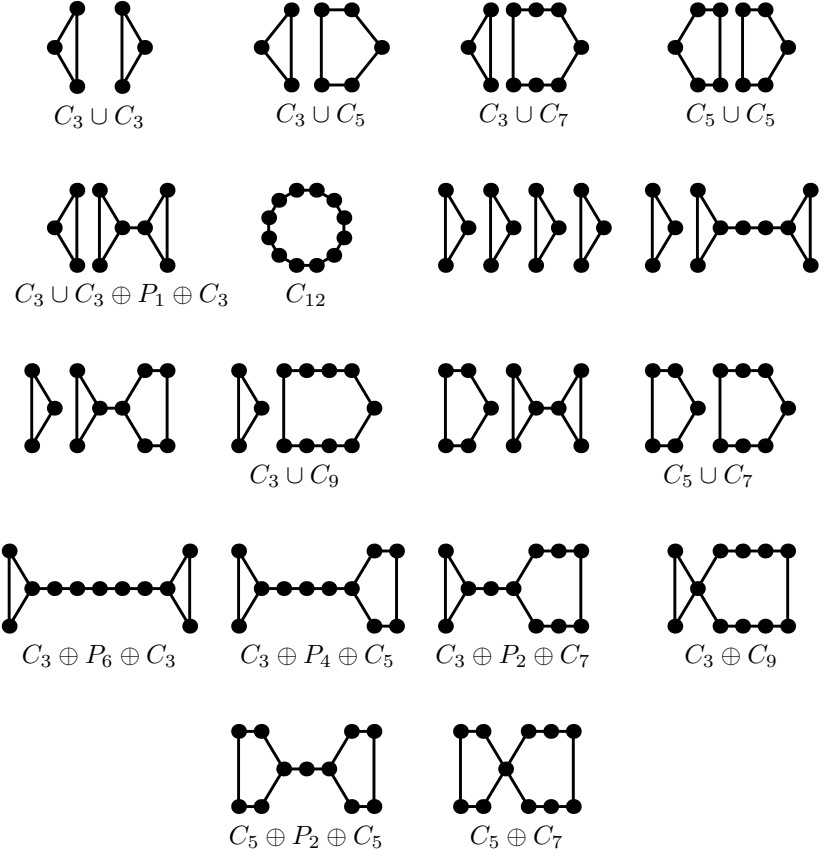}
\end{center}
\caption{The exceptional graphs in the proof of Lemma~\ref{lm:deck12-det00}.}
\label{fig:deck12-det00}
\end{figure}

\begin{proof}
Fix a $12$-deck $\DD$ with the properties as supposed in the statement of the lemma.
Note that $s_{\DD}(C_{\ell})=0$ for $\ell=4,6,8,10$ and
$s_{\DD}(C_3\oplus C_{\ell})=0$ for $\ell=5,7$ by Corollary~\ref{cor:deck10-III-classify}.
This together with the assumptions of the lemma implies that
$s_{\DD}(H)=0$ for all principal graphs $H$ with at most twelve edges with the following exceptions:
$C_3\cup C_3$, $C_3\cup C_5$, $C_3\cup C_7$, $C_5\cup C_5$,
$C_3\cup C_3\oplus P_1\oplus C_3$, $C_{12}$, $C_3\cup C_3\cup C_3\cup C_3$,
$C_3\cup C_3\oplus P_3\oplus C_3$, $C_3\cup C_3\oplus P_1\oplus C_5$, $C_3\cup C_9$, $C_5\cup C_3\oplus P_1\oplus C_3$,
$C_5\cup C_7$, $C_3\oplus P_6\oplus C_3$, $C_3\oplus P_4\oplus C_5$, $C_3\oplus P_2\oplus C_7$,
$C_3\oplus C_9$, $C_5\oplus P_2\oplus C_5$ and $C_5\oplus C_7$.
The graphs are depicted in Figure~\ref{fig:deck12-det00}.

We first show that if the deck $\DD$ satisfies one of the six conditions in the statement of the lemma,
then $\DD$ is of Class II.
We will construct four different kernels using Lemma~\ref{lm:core}.
First, we apply Lemma~\ref{lm:core} with $k=11$, $m=1$, any $\delta\in (0,1/2)$ such that $s_{\DD}(C_{12})\delta\le 1/2$,
$\gamma_3=\gamma_5=\gamma_7=\gamma_9=\gamma_{11}=0$, $\sigma_1=\sqrt{\delta}$ (note that $\sigma_1^{12} \le \delta/2$),
$\tau_{1,3}=1$, $\tau_{1,7}=\tau_{1,9}=-(1+s_{\DD}(C_{12})+s_{\DD}(C_3\oplus P_6\oplus C_3))\delta^{-1}$ and $\tau_{1,5}=\tau_{1,11}=0$
to obtain a non-zero kernel $U$ with the properties given in Lemma~\ref{lm:core}.
Observe that $t(H,U)=0$ for the graphs depicted in Figure~\ref{fig:deck12-det00}
unless $H$ is $C_{12}$, $C_3\oplus P_6\oplus C_3$, $C_3\oplus P_2\oplus C_7$ or $C_3\oplus C_9$.
It follows that $c^U_{\DD,2}=\cdots=c^U_{\DD,10}=0$ and,
\begin{align*}
c^U_{\DD,12} & = \sum_{H,\|H\|=12} s_{\DD}(H) t(H,U) \\
	     & \le \sigma_1^6 \tau_{1,3}^2 s_{\DD}(C_3\oplus P_6\oplus C_3)  +  \sigma_1^2 \tau_{1,3} \tau_{1,7} s_{\DD}(C_3\oplus P_2\oplus C_7) \\
	     & + \tau_{1,3} \tau_{1,9} s_{\DD}(C_3\oplus C_9) + \delta  s_{\DD}(C_{12}) \\
	     & \le s_{\DD}(C_3\oplus P_6\oplus C_3)-(1+s_{\DD}(C_{12})+s_{\DD}(C_3\oplus P_6\oplus C_3))+s_{\DD}(C_{12})\\
	     & \le -1\,.
\end{align*}
Hence, the deck $\DD$ is of Class II if the first or second condition in the statement applies.

We next analyze the case when the third condition applies.
If $4s_{\DD}(C_3\oplus P_6\oplus C_3)s_{\DD}(C_5\oplus P_2\oplus C_5)<s_{\DD}(C_3\oplus P_4\oplus C_5)^2$,
then there exists a vector $(z_3,z_5)\in\RR^2$ such that
\[\begin{pmatrix} z_3 \\ z_5 \end{pmatrix}^T
  \begin{pmatrix}
  s_{\DD}(C_3\oplus P_6\oplus C_3) & \frac{s_{\DD}(C_3\oplus P_4\oplus C_5)}{2} \\
  \frac{s_{\DD}(C_3\oplus P_4\oplus C_5)}{2} & s_{\DD}(C_5\oplus P_2\oplus C_5)
  \end{pmatrix}
  \begin{pmatrix} z_3 \\ z_5 \end{pmatrix}=-1.\]
We apply Lemma~\ref{lm:core} with $k=11$, $m=1$, any $\delta\in (0,1)$ such that $s_{\DD}(C_{12}) \delta \le 1/2$,
$\gamma_3=\gamma_5=\gamma_7=\gamma_9=\gamma_{11}=0$, $\sigma_1=\delta/2$, $\tau_{1,3}=8z_3/\delta^3$, $\tau_{1,5}=2z_5/\delta$ and $\tau_{1,7}=\tau_{1,9}=\tau_{1,11}=0$
to get a non-zero kernel $U$ with the properties given in Lemma~\ref{lm:core}.
Let $f_1$ be the eigenfunction associated with the eigenvalue $\sigma_1$ and
note that $t^{C_3 \oplus P_3}_U(x)=\sigma_1^3 \tau_{1,3} f_1(x)=z_3f_1(x)$ 
and $t^{C_5 \oplus P_1}_U(x)=\sigma_1 \tau_{1,5} f_1(x)=z_5f_1(x)$.
Observe that $t(H,U)=0$ for the graphs depicted in Figure~\ref{fig:deck12-det00}
unless $H$ is $C_{12}$, $C_3\oplus P_6\oplus C_3$, $C_3\oplus P_4\oplus C_5$ or $C_5\oplus P_2\oplus C_5$.
It follows that $c^U_{\DD,2}=\cdots=c^U_{\DD,10}=0$ and
\begin{align*}
 c^U_{\DD,12} =& s_{\DD}(C_{12})t(C_{12},U)+\\
              & \int_{[0,1]}
                \begin{pmatrix} z_3 \\ z_5 \end{pmatrix}^T
                \begin{pmatrix}
                s_{\DD}(C_3\oplus P_6\oplus C_3) & \frac{s_{\DD}(C_3\oplus P_4\oplus C_5)}{2} \\
                \frac{s_{\DD}(C_3\oplus P_4\oplus C_5)}{2} & s_{\DD}(C_5\oplus P_2\oplus C_5)
                \end{pmatrix}
		            \begin{pmatrix} z_3 \\ z_5 \end{pmatrix} f_1(x)^2 \dd x,
\end{align*}
which is at most $-1/2$. Hence, the deck $\DD$ is of Class II.

If the fourth condition in the statement holds, 
we apply Lemma~\ref{lm:core} with $k=11$, $m=1$, any $\delta\in (0,1)$ such that $s_{\DD}(C_{12})\delta\le 1/2$,
$\gamma_3=\gamma_5=\gamma_7=\gamma_9=\gamma_{11}=0$, $\sigma_1=0$,
$\tau_{1,3}=\tau_{1,9}=\tau_{1,11}=0$, $\tau_{1,5}=1$ and $\tau_{1,7}=-1$
to obtain a non-zero kernel $U$ with the properties given in Lemma~\ref{lm:core}.
Observe that $t(H,U)=0$ for the graphs depicted in Figure~\ref{fig:deck12-det00}
unless $H$ is $C_{12}$ or  $C_5\oplus C_7$. 
It follows that $c^U_{\DD,2}=\cdots=c^U_{\DD,10}=0$, and
$c^U_{\DD,12} \leq -1/2$, so we again conclude that the $12$-deck $\DD$ is of Class II.

It remains to analyze the last two conditions given in the statement of the lemma.
We apply Lemma~\ref{lm:core} with $k=11$, $m=1$, any $\delta\in (0,1/2)$ such that $s_{\DD}(C_{12})\delta\le 1/2$,
$\gamma_3=\gamma_9=\gamma_{11}=0$, $\gamma_5=-(1+s_{\DD}(C_{12})+s_{\DD}(C_3\oplus P_6\oplus C_3))\delta^{-1}$, $\gamma_7=1$, $\sigma_1=\delta$, $\tau_{1,3}=1$ and $\tau_{1,5}=\tau_{1,7}=\tau_{1,9}=\tau_{1,11}=0$
to get a non-zero kernel $U$ with the properties given in Lemma~\ref{lm:core}.
Observe that $t(H,U)=0$ for the graphs depicted in Figure~\ref{fig:deck12-det00}
unless $H$ is $C_{12}$, $C_5\cup C_5$, $C_5\cup C_3\oplus P_1\oplus C_3$, $C_5\cup C_7$ or $C_3\oplus P_6\oplus C_3$.
It follows that $c^U_{\DD,2}=\cdots=c^U_{\DD,10}=0$ (note that $s_{\DD}(C_5\cup C_5)=0$) and,
\begin{align*}
c^U_{\DD,12} & = \sum_{H,\|H\|=12} s_{\DD}(H) t(H,U) \\ 
	     & \le \sigma_1^6 \tau_{1,3}^2 s_{\DD}(C_3\oplus P_6\oplus C_3) + \sigma_1 \gamma_5 \tau_{1,3}^2 s_{\DD}(C_5\cup C_3\oplus P_1\oplus C_3) \\
	     & + \gamma_5 \gamma_7  s_{\DD}(C_5\cup C_7)  + \delta  s_{\DD}(C_{12}) \\
	     & \le s_{\DD}(C_3\oplus P_6\oplus C_3)-(1+s_{\DD}(C_{12})+s_{\DD}(C_3\oplus P_6\oplus C_3))+s_{\DD}(C_{12}) \\
	     & \le -1\,.
\end{align*}
We conclude that the deck $\DD$ is of Class II.

We now prove that if the deck $\DD$ does not satisfy any of the six conditions in the statement of the lemma,
then it is of Class I or Class III.
Following the line of arguments presented in the proofs of Lemmas~\ref{lm:deck8-I} and~\ref{lm:deck10-I},
it is enough to establish that the deck $\DD$ is of Class III when $s_{\DD}(C_{12})=0$;
note that if $s_{\DD}(C_{12})=0$, then the deck $\DD$ is not of Class I by Lemma~\ref{lm:core-III}.

Let $U$ be an arbitrary non-zero kernel.
The coefficient $c^U_{\DD,2}$ is positive unless $U$ is balanced.
Hence, we can assume that $U$ is balanced.
Inspecting the graphs in Figure~\ref{fig:deck12-det00},
we obtain that $c^U_{\DD,4}=0$ and $c^U_{\DD,6}\ge 0$, and
the equality holds only if $t(C_3\cup C_3,U)=t(C_3,U)^2=0$ since $s_{\DD}(C_3\cup C_3)>0$.
Hence, we can assume that $t(C_3,U)=0$ in the rest.
We next obtain, again inspecting the graphs in Figure~\ref{fig:deck12-det00}, that
$c^U_{\DD,8}=0$ and $c^U_{\DD,10}\ge 0$, and the equality
can hold only if $s_{\DD}(C_5\cup C_5)=0$ or $t(C_5,U)=0$;
if $s_{\DD}(C_5\cup C_5)=0$ or $t(C_5,U)=0$,
it holds that $s_{\DD}(H)t(H,U)=0$ if $H$ is $C_5\cup C_3\oplus P_1\oplus C_3$ or $C_5\cup C_7$ (here, we use that $\DD$
does not satisfy the last two conditions in the statement of the lemma).
In particular, if $c^U_{\DD,10}=0$,
then $s_{\DD}(H)t(H,U)$ can be non-zero only for the following three principal $12$-edge graphs:
$C_3\oplus P_6\oplus C_3$, $C_3\oplus P_4\oplus C_5$ and $C_5\oplus P_2\oplus C_5$.
It follows that
\[c^U_{\DD,12} = \int_{[0,1]}
                \mbox{\scalebox{0.90}{$
                \begin{pmatrix} t_U^{C_3\oplus P_3}(x) \\ t_U^{C_5\oplus P_1}(x) \end{pmatrix}^T
                \begin{pmatrix}
                s_{\DD}(C_3\oplus P_6\oplus C_3) & \frac{s_{\DD}(C_3\oplus P_4\oplus C_5)}{2} \\
                \frac{s_{\DD}(C_3\oplus P_4\oplus C_5)}{2} & s_{\DD}(C_5\oplus P_2\oplus C_5)
                \end{pmatrix}
                \begin{pmatrix} t_U^{C_3\oplus P_3}(x) \\ t_U^{C_5\oplus P_1}(x) \end{pmatrix}$}} \dd x,\]
which is non-negative since the matrix is positive semidefinite (here, we use that
$4s_{\DD}(C_3\oplus P_6\oplus C_3)s_{\DD}(C_5\oplus P_2\oplus C_5)\ge s_{\DD}(C_3\oplus P_4\oplus C_5)^2$).
We conclude that the deck $\DD$ is of Class III.
\end{proof}

\begin{lemma}
\label{lm:deck12-det0+}
Let $\DD$ be a $12$-deck such that its $10$-deck is of Class III, $s_{\DD}(P_2)>0$, $s_{\DD}(C_3\oplus C_3)=s_{\DD}(C_3\oplus P_2\oplus C_3)=0$ and $s_{\DD}(C_3\cup C_3)>0$.
Further suppose that $s_{\DD}(C_3\oplus P_4\oplus C_3)=s_{\DD}(C_3\oplus P_2\oplus C_5)=0$ and $s_{\DD}(C_5\oplus C_5)>0$.
If $s_{\DD}(C_3\oplus C_9)>0$ or $s_{\DD}(C_3\oplus P_2\oplus C_7)>0$, then $\DD$ is of Class II.
Otherwise, $\DD$ is of Class I if $s_{\DD}(C_{12})>0$ and of Class III if $s_{\DD}(C_{12})=0$.
\end{lemma}

\begin{proof}
Fix a $12$-deck $\DD$ with the properties as supposed in the statement of the lemma.
Note that $s_{\DD}(C_{\ell})=0$ for $\ell=4,6,8,10$ and
$s_{\DD}(C_3\oplus C_{\ell})=0$ for $\ell=5,7$ by Corollary~\ref{cor:deck10-III-classify}.
We first show that if $s_{\DD}(C_3\oplus C_9)>0$ or $s_{\DD}(C_3\oplus P_2\oplus C_7)>0$, then $\DD$ is of Class II.
We use the same application of Lemma~\ref{lm:core} as we did in the proof of Lemma~\ref{lm:deck12-det00} with one of the first two conditions, that is
we apply Lemma~\ref{lm:core} with $k=11$, $m=1$, any $\delta\in (0,1/2)$ such that $s_{\DD}(C_{12})\delta\le 1/2$,
$\gamma_3=\gamma_5=\gamma_7=\gamma_9=\gamma_{11}=0$, $\sigma_1=\sqrt{\delta}$ (note that $\sigma_1^{12} \le \delta/2$),
$\tau_{1,3}=1$, $\tau_{1,7}=\tau_{1,9}=-(1+s_{\DD}(C_{12})+s_{\DD}(C_3\oplus P_6\oplus C_3))\delta^{-1}$ and $\tau_{1,5}=\tau_{1,11}=0$.
The same calculations prove that the $12$-deck $\DD$ is of Class II
if $s_{\DD}(C_3\oplus C_9)>0$ or $s_{\DD}(C_3\oplus P_2\oplus C_7)>0$.

We now prove that if $s_{\DD}(C_3\oplus C_9)=0$ and $s_{\DD}(C_3\oplus P_2\oplus C_7)=0$,
then the deck $\DD$ satisfying the assumptions of the lemma is of Class I or Class III.
Following the line of arguments presented in the proofs of Lemmas~\ref{lm:deck8-I} and~\ref{lm:deck10-I},
it is enough to establish that the deck $\DD$ is of Class III when $s_{\DD}(C_{12})=0$;
note that if $s_{\DD}(C_{12})=0$, then the deck $\DD$ is not of Class I by Lemma~\ref{lm:core-III}.

Let $U$ be an arbitrary non-zero kernel.
The coefficient $c^U_{\DD,2}$ is positive unless $U$ is balanced.
Hence, we can assume that $U$ is balanced.
It follows that $c^U_{\DD,4}=0$ and $c^U_{\DD,6}=s_{\DD}(C_3\cup C_3)t(C_3,U)^2$.
In particular, either $c^U_{\DD,6}$ is positive or $t(C_3,U)=0$;
we focus on the latter case in the rest of the proof.
Observe that every principal graph $H$ with at most ten edges satisfies that $s_{\DD}(H)=0$ or $t(H,U)=0$
unless $H$ is $C_5\cup C_5$ or $C_5\oplus C_5$.
It follows that $c^U_{\DD,8}=0$ and $c^U_{\DD,10}\ge s_{\DD}(C_5\oplus C_5)t(C_5\oplus C_5,U)$.
In particular, the coefficient $c^U_{\DD,10}$ is positive unless $t^{C_5}_U(x)=0$ for almost every $x\in [0,1]$.
If $t^{C_5}_U(x)=0$ for almost every $x\in [0,1]$,
then $s_{\DD}(H)=0$ or $t(H,U)=0$ for every $12$-edge principal graph $H$ different from $C_3\oplus P_6\oplus C_3$.
It follows that if $c^U_{\DD,10}=0$, then $c^U_{\DD,12}=s_{\DD}(C_3\oplus P_6\oplus C_3)t(C_3\oplus P_6\oplus C_3,U)\ge 0$.
We conclude that the $12$-deck $\DD$ is of Class III.
\end{proof}

\begin{lemma}
\label{lm:deck12-det+0}
Let $\DD$ be a $12$-deck such that its $10$-deck is of Class III, $s_{\DD}(P_2)>0$, $s_{\DD}(C_3\oplus C_3)=s_{\DD}(C_3\oplus P_2\oplus C_3)=0$ and $s_{\DD}(C_3\cup C_3)>0$.
Further suppose that $s_{\DD}(C_5\oplus C_5)=s_{\DD}(C_3\oplus P_2\oplus C_5)=0$ and $s_{\DD}(C_3\oplus P_4\oplus C_3)>0$.
If 
\begin{itemize}
\item $s_{\DD}(C_3\oplus C_9)>0$,
\item $s_{\DD}(C_5\oplus C_7)>0$, or
\item $s_{\DD}(C_5\cup C_5)=0$ and $s_{\DD}(C_5\cup C_7)>0$,
\end{itemize}
then $\DD$ is of Class II.
Otherwise, $\DD$ is of Class I if $s_{\DD}(C_{12})>0$ and of Class III if $s_{\DD}(C_{12})=0$.
\end{lemma}

\begin{proof}
Fix a $12$-deck $\DD$ with the properties as supposed in the statement of the lemma.
Note that $s_{\DD}(C_{\ell})=0$ for $\ell=4,6,8,10$ and
$s_{\DD}(C_3\oplus C_{\ell})=0$ for $\ell=5,7$ by Corollary~\ref{cor:deck10-III-classify}.
We first show that if $s_{\DD}(C_3\oplus C_9)>0$ or $s_{\DD}(C_5\oplus C_7)>0$, then $\DD$ is of Class II.
Apply Lemma~\ref{lm:core} with $k=11$, $m=1$, any $\delta\in (0,1)$ such that $s_{\DD}(C_{12})\delta\le 1/2$,
$\gamma_3=\gamma_5=\gamma_7=\gamma_9=\gamma_{11}=0$, $\sigma_1=0$, $\tau_{1,3}=\tau_{1,5}=1$, $\tau_{1,7}=\tau_{1,9}=-1$ and $\tau_{1,11}=0$,
to get a non-zero kernel $U$ with the properties listed in Lemma~\ref{lm:core}.
Every principal graph $H$ with at most twelve edges satisfies that $s_{\DD}(H)=0$ or $t(H,U)=0$
unless $H$ is $C_{12}$, $C_3\oplus C_9$ or $C_5\oplus C_7$.
Hence, it holds that $c^U_{\DD,2}=\cdots=c^U_{\DD,10}=0$ and $c^U_{\DD,12}\le -1/2$,
which yields that the $12$-deck $\DD$ is of Class II.

We next show that if $s_{\DD}(C_5\cup C_5)=0$ and $s_{\DD}(C_5\cup C_7)>0$, then $\DD$ is also of Class II.
Apply Lemma~\ref{lm:core} with $k=11$, $m=0$, any $\delta\in (0,1)$ such that $s_{\DD}(C_{12})\delta\le 1/2$,
$\gamma_5=1$, $\gamma_7=-1$, $\gamma_3=\gamma_9=\gamma_{11}=0$,
to get a non-zero kernel $U$ with the properties listed in Lemma~\ref{lm:core}.
Every principal graph $H$ with at most twelve edges satisfies that $s_{\DD}(H)=0$ or $t(H,U)=0$
unless $H$ is $C_{12}$ or $C_5\cup C_7$.
It follows that $c^U_{\DD,2}=\cdots=c^U_{\DD,10}=0$ and $c^U_{\DD,12}\le -1/2$ and
we again conclude that the $12$-deck $\DD$ is of Class II.

We now prove that if the deck $\DD$ does not satisfy any of the three conditions in the statement of the lemma,
then it is of Class I or Class III.
Following the line of arguments presented in the proofs of Lemmas~\ref{lm:deck8-I} and~\ref{lm:deck10-I},
it is enough to establish that the deck $\DD$ is of Class III when $s_{\DD}(C_{12})=0$;
note that if $s_{\DD}(C_{12})=0$, then the deck $\DD$ is not of Class I by Lemma~\ref{lm:core-III}.

Let $U$ be an arbitrary non-zero kernel.
The coefficient $c^U_{\DD,2}$ is positive unless $U$ is balanced.
So, we can assume that $U$ is balanced.
It follows that $c^U_{\DD,4}=0$ and $c^U_{\DD,6}=s_{\DD}(C_3\cup C_3)t(C_3\cup C_3,U)\ge 0$ and
the equality holds only if $t(C_3,U)=0$.
Hence, we assume that $t(C_3,U)=0$ in the rest of the proof.
It follows that $c^U_{\DD,8}=0$ and
\[c^U_{\DD,10}=s_{\DD}(C_5\cup C_5)t(C_5\cup C_5,U)+s_{\DD}(C_3\oplus P_4\oplus C_3)t(C_3\oplus P_4\oplus C_3,U).\]
In particular, $c^U_{\DD,10}$ is non-negative and
if $c^U_{\DD,10}=0$, then $t^{C_3\oplus P_2}_U(x)=0$ for almost every $x\in [0,1]$ and
either $s_{\DD}(C_5\cup C_5)=0$ or $t(C_5,U)=0$ or both.
Since $t^{C_3\oplus P_2}_U=U^2t^{C_3}_U$ and $t^{C_3\oplus P_1}_U=Ut^{C_3}_U$,
it follows that $t^{C_3\oplus P_1}_U(x)=0$ for almost every $x\in [0,1]$.
Hence, if $c^U_{\DD,10}=0$ and none of the three conditions in the statement of the lemma applies,
we get that $s_{\DD}(H)=0$ or $t(H,U)=0$ for every $12$-edge principal graph $H$ unless $H$ is $C_5\oplus P_2\oplus C_5$.
It follows that $c^U_{\DD,12}\ge 0$ and so the $12$-deck $\DD$ is of Class III.
\end{proof}

The final two lemmas analyze the case when
the quantity $4s_{\DD}(C_3\oplus P_4\oplus C_3)s_{\DD}(C_5\oplus C_5)-s_{\DD}(C_3\oplus P_2\oplus C_5)^2$ is non-negative and
both $s_{\DD}(C_3\oplus P_4\oplus C_3)$ and $s_{\DD}(C_5\oplus C_5)$ are positive.

\begin{lemma}
\label{lm:deck12-det++}
Let $\DD$ be a $12$-deck such that its $10$-deck is of Class III, $s_{\DD}(P_2)>0$, $s_{\DD}(C_3\oplus C_3)=s_{\DD}(C_3\oplus P_2\oplus C_3)=0$ and $s_{\DD}(C_3\cup C_3)>0$.
Further suppose that $s_{\DD}(C_3\oplus P_4\oplus C_3)>0$, $s_{\DD}(C_5\oplus C_5)>0$ and 
$4s_{\DD}(C_3\oplus P_4\oplus C_3)s_{\DD}(C_5\oplus C_5)=s_{\DD}(C_3\oplus P_2\oplus C_5)^2$.
If 
\begin{itemize}
\item $s_{\DD}(C_3\oplus C_9)>0$,
\item $A=\mbox{\scalebox{0.70}{$
    \begin{pmatrix} -2s_{\DD}(C_5\oplus C_5) \\ s_{\DD}(C_3\oplus P_2\oplus C_5) \end{pmatrix}^T
    \begin{pmatrix} s_{\DD}(C_3\oplus P_6\oplus C_3) & \frac{s_{\DD}(C_3\oplus P_4\oplus C_5)}{2} \\ \frac{s_{\DD}(C_3\oplus P_4\oplus C_5)}{2} & s_{\DD}(C_5\oplus P_2\oplus C_5) \end{pmatrix}
    \begin{pmatrix} -2s_{\DD}(C_5\oplus C_5) \\ s_{\DD}(C_3\oplus P_2\oplus C_5) \end{pmatrix}$}}
      <0$, or
\item $-2s_{\DD}(C_3\oplus P_2\oplus C_7)s_{\DD}(C_5\oplus C_5)+s_{\DD}(C_5\oplus C_7)s_{\DD}(C_3\oplus P_2\oplus C_5)\not=0$,
\end{itemize}
then $\DD$ is of Class II.
Otherwise, $\DD$ is of Class I if $s_{\DD}(C_{12})>0$ and of Class III if $s_{\DD}(C_{12})=0$.
\end{lemma}

\begin{figure}
\begin{center}
\epsfbox{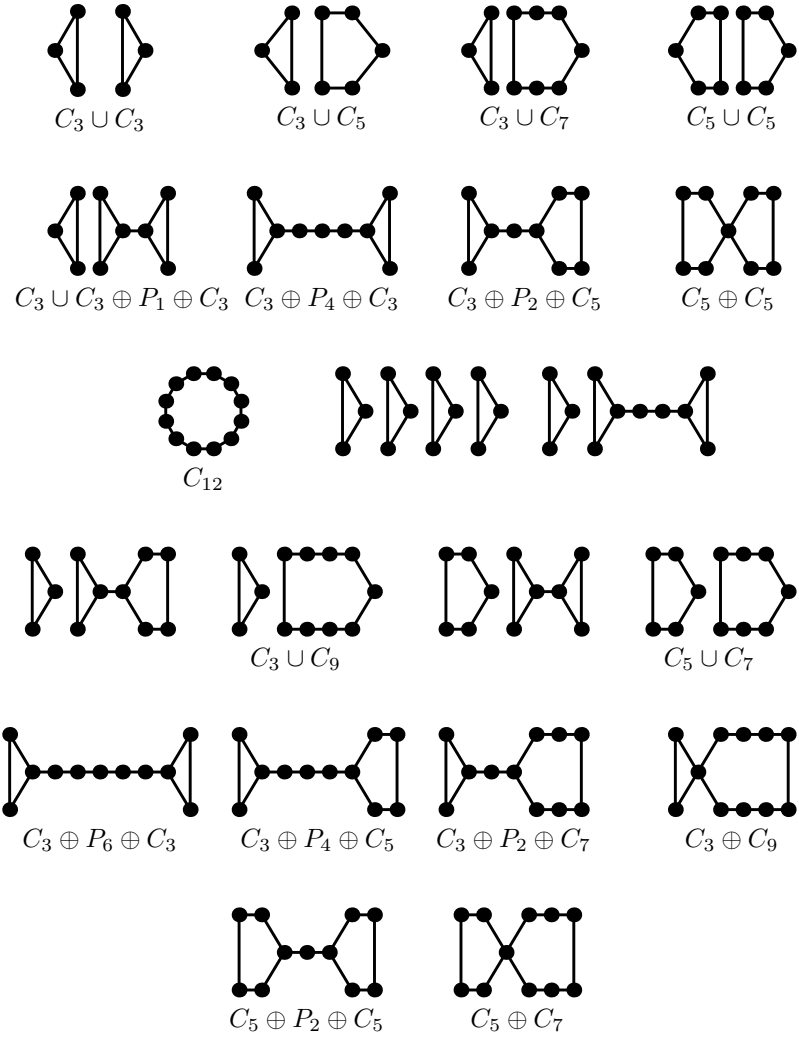}
\end{center}
\caption{The exceptional graphs in the proof of Lemma~\ref{lm:deck12-det++}.}
\label{fig:deck12-det++}
\end{figure}

\begin{proof}
Fix a $12$-deck $\DD$ with the properties as supposed in the statement of the lemma.
Note that $s_{\DD}(C_{\ell})=0$ for $\ell=4,6,8,10$ and
$s_{\DD}(C_3\oplus C_{\ell})=0$ for $\ell=5,7$ by Corollary~\ref{cor:deck10-III-classify}.
This together with the assumptions of the lemma implies that
$s_{\DD}(H)=0$ for all principal graphs $H$ with at most twelve edges with the following exceptions:
$C_3\cup C_3$, $C_3\cup C_5$, $C_3\cup C_7$, $C_5\cup C_5$,
$C_3\cup C_3\oplus P_1\oplus C_3$, 
$C_3\oplus P_4\oplus C_3$, $C_3\oplus P_2\oplus C_5$, $C_5\oplus C_5$,
$C_{12}$, $C_3\cup C_3\cup C_3\cup C_3$,
$C_3\cup C_3\oplus P_3\oplus C_3$, $C_3\cup C_3\oplus P_1\oplus C_5$, $C_3\cup C_9$, $C_5\cup C_3\oplus P_1\oplus C_3$,
$C_5\cup C_7$, $C_3\oplus P_6\oplus C_3$, $C_3\oplus P_4\oplus C_5$, $C_3\oplus P_2\oplus C_7$,
$C_3\oplus C_9$, $C_5\oplus P_2\oplus C_5$ and $C_5\oplus C_7$.
The graphs are depicted in Figure~\ref{fig:deck12-det++}.
Also note that the matrix 
\[\begin{pmatrix}
  s_{\DD}(C_3\oplus P_4\oplus C_3) & \frac{s_{\DD}(C_3\oplus P_2\oplus C_5)}{2} \\
  \frac{s_{\DD}(C_3\oplus P_2\oplus C_5)}{2} & s_{\DD}(C_5\oplus C_5)
  \end{pmatrix}\]
has rank one and so there exists a non-zero vector $(z_3,z_5)\in\RR^2$ such that
\[\begin{pmatrix}
  s_{\DD}(C_3\oplus P_4\oplus C_3) & \frac{s_{\DD}(C_3\oplus P_2\oplus C_5)}{2} \\
  \frac{s_{\DD}(C_3\oplus P_2\oplus C_5)}{2} & s_{\DD}(C_5\oplus C_5)
  \end{pmatrix}
  \begin{pmatrix} z_3 \\ z_5 \end{pmatrix}=
  \begin{pmatrix} 0 \\ 0\end{pmatrix}.\]
Note that the vectors $(z_3,z_5)$ and $\left(-s_{\DD}(C_5\oplus C_5),\frac{s_{\DD}(C_3\oplus P_2\oplus C_5)}{2}\right)$
are non-zero multiples of each other,
in particular, both $z_3\not=0$ and $z_5\not=0$.  
By multiplying $(z_3,z_5)$ by a suitable constant, we may assume without loss of generality that
\[\begin{pmatrix} z_3 \\ z_5 \end{pmatrix}^T
  \begin{pmatrix} s_{\DD}(C_3\oplus P_6\oplus C_3) & \frac{s_{\DD}(C_3\oplus P_4\oplus C_5)}{2} \\ \frac{s_{\DD}(C_3\oplus P_4\oplus C_5)}{2} & s_{\DD}(C_5\oplus P_2\oplus C_5) \end{pmatrix}
  \begin{pmatrix} z_3 \\ z_5 \end{pmatrix}\]
is equal to $-1$, $0$ or $+1$.  

We first show that if $s_{\DD}(C_3\oplus C_9)>0$, then $\DD$ is of Class II.
Apply Lemma~\ref{lm:core} with $k=11$, $m=1$, any $\delta\in (0,1)$ such that $s_{\DD}(C_{12})\delta\le 1/2$,
$\gamma_3=\gamma_5=\gamma_7=\gamma_9=\gamma_{11}=0$, $\sigma_1=0$, $\tau_{1,3}=1$, $\tau_{1,5}=\tau_{1,7}=\tau_{1,11}=0$ and $\tau_{1,9}=-1$
to get a non-zero kernel $U$ with the properties listed in Lemma~\ref{lm:core}.
Every principal graph $H$ with at most twelve edges satisfies that $s_{\DD}(H)=0$ or $t(H,U)=0$
unless $H$ is $C_{12}$ or $C_3\oplus C_9$.
Hence, it holds that $c^U_{\DD,2}=\cdots=c^U_{\DD,10}=0$ and $c^U_{\DD,12}\le -1/2$,
which yields that the $12$-deck $\DD$ is of Class II.

In the second and third cases described in the statement of the lemma,
we apply Lemma~\ref{lm:core} with $k=11$, $m=1$, any $\delta\in (0,1/2)$ such that $s_{\DD}(C_{12})\delta\le 1/2$,
$\gamma_3=\gamma_5=\gamma_7=\gamma_9=\gamma_{11}=0$, $\sigma_1=\delta^{1/4}$, $\tau_{1,3}=z_3/\delta^{1/2}$, $\tau_{1,5}=z_5$, $\tau_{1,9}=\tau_{1,11}=0$, and
$\tau_{1,7}=-2\delta^{1/2}(z_3s_{\DD}(C_3\oplus P_2\oplus C_7)+z_5s_{\DD}(C_5\oplus C_7))^{-1}$
if $z_3s_{\DD}(C_3\oplus P_2\oplus C_7)+z_5s_{\DD}(C_5\oplus C_7)$ is non-zero, and $\tau_{1,7}=0$, otherwise,
to get a non-zero kernel $U$ with the properties listed in Lemma~\ref{lm:core}.
Let $f_1$ be the eigenfunction of $U$ associated with $\sigma_1$.
Every principal graph $H$ with at most twelve edges satisfies that $s_{\DD}(H)=0$ or $t(H,U)=0$
unless $H$ is $C_3\oplus P_4\oplus C_3$, $C_3\oplus P_2\oplus C_5$, $C_5\oplus C_5$,
$C_{12}$, $C_3\oplus P_6\oplus C_3$, $C_3\oplus P_4\oplus C_5$, $C_5\oplus P_2\oplus C_5$,
$C_3\oplus P_2\oplus C_7$ and $C_5\oplus C_7$.
It follows that $c^U_{\DD,2}=\cdots=c^U_{\DD,8}=0$ and
\begin{align*}
 c^U_{\DD,10}&= \int_{[0,1]}
                \begin{pmatrix}t_U^{C_3\oplus P_2}(x) \\ t_U^{C_5}(x) \end{pmatrix}^T
                \begin{pmatrix}
                s_{\DD}(C_3\oplus P_4\oplus C_3) & \frac{s_{\DD}(C_3\oplus P_2\oplus C_5)}{2} \\
                \frac{s_{\DD}(C_3\oplus P_2\oplus C_5)}{2} & s_{\DD}(C_5\oplus C_5)
                \end{pmatrix}
                \begin{pmatrix}t_U^{C_3\oplus P_2}(x) \\ t_U^{C_5}(x) \end{pmatrix}\dd x\\
             &=\int_{[0,1]}
                \begin{pmatrix}z_3 \\ z_5 \end{pmatrix}^T
                \begin{pmatrix}
                s_{\DD}(C_3\oplus P_4\oplus C_3) & \frac{s_{\DD}(C_3\oplus P_2\oplus C_5)}{2} \\
                \frac{s_{\DD}(C_3\oplus P_2\oplus C_5)}{2} & s_{\DD}(C_5\oplus C_5)
                \end{pmatrix}
                \begin{pmatrix}z_3 \\ z_5 \end{pmatrix}f_1(x)^2\dd x=0.
\end{align*}
We next analyze the coefficient $c^U_{\DD,12}$.
First note that the sum of the terms $s_{\DD}(H)t(H,U)$ for $H$ being one of the graphs
$C_3\oplus P_6\oplus C_3$, $C_3\oplus P_4\oplus C_5$ and $C_5\oplus P_2\oplus C_5$
is equal to
\[B_1=\int_{[0,1]}
    \begin{pmatrix} z_3 \\ z_5 \end{pmatrix}^T
    \begin{pmatrix} s_{\DD}(C_3\oplus P_6\oplus C_3) & \frac{s_{\DD}(C_3\oplus P_4\oplus C_5)}{2} \\ \frac{s_{\DD}(C_3\oplus P_4\oplus C_5)}{2} & s_{\DD}(C_5\oplus P_2\oplus C_5) \end{pmatrix}
    \begin{pmatrix}z_3 \\ z_5 \end{pmatrix}\sigma_1^2 f_1(x)^2\dd x.\]
Observe that $B_1=\delta^{1/2}$ if $A>0$, $B_1=0$ if $A=0$ and $B_1=-\delta^{1/2}$ if $A<0$.
The sum of the terms $s_{\DD}(H)t(H,U)$ for $H$ being $C_3\oplus P_2\oplus C_7$ and $C_5\oplus C_7$
is equal to
\[B_2=\int_{[0,1]}(s_{\DD}(C_3\oplus P_2\oplus C_7)z_3+s_{\DD}(C_5\oplus C_7)z_5)\tau_{1,7} f_1(x)^2\dd x.\]
If $z_3s_{\DD}(C_3\oplus P_2\oplus C_7)+z_5s_{\DD}(C_5\oplus C_7)$ is non-zero, then $B_2=-2\delta^{1/2}$;
otherwise, $B_2=0$.
Since $s_{\DD}(C_{12})t(C_{12},U)\le \delta$,
it follows that $c^U_{\DD,12}\le B_1+B_2+\delta\le -\delta^{1/2}+\delta<0$ and so the $12$-deck $\DD$ is of Class II.

We now prove that if the deck $\DD$ does not satisfy any of the three conditions in the statement of the lemma,
then it is of Class I or Class III.
Following the line of arguments presented in the proofs of Lemmas~\ref{lm:deck8-I} and~\ref{lm:deck10-I},
it is enough to establish that the deck $\DD$ is of Class III when $s_{\DD}(C_{12})=0$;
note that if $s_{\DD}(C_{12})=0$, then the deck $\DD$ is not of Class I by Lemma~\ref{lm:core-III}.

Let $U$ be an arbitrary non-zero kernel.
The coefficient $c^U_{\DD,2}$ is positive unless $U$ is balanced.
So, we can assume that $U$ is balanced.
It follows that $c^U_{\DD,4}=0$ and $c^U_{\DD,6}=s_{\DD}(C_3\cup C_3)t(C_3\cup C_3,U)\ge 0$ and
the equality holds only if $t(C_3,U)=0$.
Hence, we assume that $t(C_3,U)=0$ in the rest of the proof.
This implies that $c^U_{\DD,8}=0$.
Observe next that
\begin{align*}		
  c^U_{\DD,10}&=s_{\DD}(C_5\cup C_5)t(C_5\cup C_5,U)\\
              &+\int_{[0,1]}
                \begin{pmatrix}t_U^{C_3\oplus P_2}(x) \\ t_U^{C_5}(x) \end{pmatrix}^T
                \begin{pmatrix}
                s_{\DD}(C_3\oplus P_4\oplus C_3) & \frac{s_{\DD}(C_3\oplus P_2\oplus C_5)}{2} \\
                \frac{s_{\DD}(C_3\oplus P_2\oplus C_5)}{2} & s_{\DD}(C_5\oplus C_5)
                \end{pmatrix}
                \begin{pmatrix}t_U^{C_3\oplus P_2}(x) \\ t_U^{C_5}(x) \end{pmatrix}\dd x.
\end{align*}		
It follows that $c^U_{\DD,10}\ge 0$ and the equality holds only if
$s_{\DD}(C_5\cup C_5)t(C_5\cup C_5,U)=0$ and
$t_U^{C_3\oplus P_2}(x)=\frac{z_3}{z_5}t_U^{C_5}(x)$ for almost every $x\in [0,1]$.
In the rest, we assume that $c^U_{\DD,10}=0$.

Since $c^U_{\DD,10}=0$, it holds that $t_U^{C_3\oplus P_2}(x)=\frac{z_3}{z_5}t_U^{C_5}(x)$ for almost every $x\in [0,1]$.
As $t_U^{C_3\oplus P_2}=U^2t_U^{C_3}$,
Proposition~\ref{prop:balanced} implies that the integral of $t_U^{C_5}$ over $[0,1]$ is zero,
i.e., $t(C_5,U)=0$.
It follows that $t(H,U)=0$ for every $12$-edge graph depicted in Figure~\ref{fig:deck12-det++} except for $C_{12}$,
$C_3\oplus P_6\oplus C_3$, $C_3\oplus P_4\oplus C_5$, $C_5\oplus P_2\oplus C_5$,
$C_3\oplus P_2\oplus C_7$ and $C_5\oplus C_7$.
Since it holds that $A\ge 0$,
i.e., the second condition of the lemma does not apply, and
the vectors $(z_3,z_5)$ and $\left(-s_{\DD}(C_5\oplus C_5),\frac{s_{\DD}(C_3\oplus P_2\oplus C_5)}{2}\right)$
are non-zero multiples of each other,
we obtain that the sum of the terms $s_{\DD}(H)t(H,U)$ for $H$ being one of the graphs
$C_3\oplus P_6\oplus C_3$, $C_3\oplus P_4\oplus C_5$ and $C_5\oplus P_2\oplus C_5$
is equal to
\begin{align*}
  & \int_{[0,1]}
    \mbox{\scalebox{0.9}{$\begin{pmatrix} Ut_U^{C_3\oplus P_2}(x) \\ Ut_U^{C_5}(x) \end{pmatrix}^T
    \begin{pmatrix} s_{\DD}(C_3\oplus P_6\oplus C_3) & \frac{s_{\DD}(C_3\oplus P_4\oplus C_5)}{2} \\ \frac{s_{\DD}(C_3\oplus P_4\oplus C_5)}{2} & s_{\DD}(C_5\oplus P_2\oplus C_5) \end{pmatrix}
    \begin{pmatrix} Ut_U^{C_3\oplus P_2}(x) \\ Ut_U^{C_5}(x) \end{pmatrix}$}}\dd x \\
= & \int_{[0,1]}
    \mbox{\scalebox{0.9}{$\begin{pmatrix} \frac{z_3}{z_5}Ut_U^{C_5}(x) \\ Ut_U^{C_5}(x) \end{pmatrix}^T
    \begin{pmatrix} s_{\DD}(C_3\oplus P_6\oplus C_3) & \frac{s_{\DD}(C_3\oplus P_4\oplus C_5)}{2} \\ \frac{s_{\DD}(C_3\oplus P_4\oplus C_5)}{2} & s_{\DD}(C_5\oplus P_2\oplus C_5) \end{pmatrix}
    \begin{pmatrix} \frac{z_3}{z_5}Ut_U^{C_5}(x) \\ Ut_U^{C_5}(x) \end{pmatrix}$}}\dd x  \\
= & \int_{[0,1]}
    \mbox{\scalebox{0.9}{$\begin{pmatrix} z_3 \\ z_5 \end{pmatrix}^T
    \begin{pmatrix} s_{\DD}(C_3\oplus P_6\oplus C_3) & \frac{s_{\DD}(C_3\oplus P_4\oplus C_5)}{2} \\ \frac{s_{\DD}(C_3\oplus P_4\oplus C_5)}{2} & s_{\DD}(C_5\oplus P_2\oplus C_5) \end{pmatrix}
    \begin{pmatrix} z_3 \\ z_5 \end{pmatrix}$}} \frac{Ut_U^{C_5}(x)^2}{z_5^2} \dd x \ge 0.
\end{align*}    
Since the vectors $(z_3,z_5)$ and $\left(-s_{\DD}(C_5\oplus C_5),\frac{s_{\DD}(C_3\oplus P_2\oplus C_5)}{2}\right)$ are multiples of each other and the third condition of the lemma does not apply,
the sum of terms $s_{\DD}(H)t(H,U)$ for $H$ being $C_3\oplus P_2\oplus C_7$ and $C_5\oplus C_7$ is equal to
\begin{align*}
& \int_{[0,1]}\left(s_{\DD}(C_3\oplus P_2\oplus C_7)t_U^{C_3\oplus P_2}(x)+s_{\DD}(C_5\oplus C_7)t_U^{C_5}(x)\right)t_U^{C_7}(x)\dd x \\
= & \frac{1}{z_5} \int_{[0,1]}\left(s_{\DD}(C_3\oplus P_2\oplus C_7)z_3+s_{\DD}(C_5\oplus C_7)z_5\right)t_U^{C_5}(x)t_U^{C_7}(x)\dd x = 0.
\end{align*}
We conclude that $c^U_{\DD,12}\ge 0$ and so the deck $\DD$ is of Class III.
\end{proof}

\begin{lemma}
\label{lm:deck12-det}
Let $\DD$ be a $12$-deck such that its $10$-deck is of Class III, $s_{\DD}(P_2)>0$, $s_{\DD}(C_3\oplus C_3)=s_{\DD}(C_3\oplus P_2\oplus C_3)=0$ and $s_{\DD}(C_3\cup C_3)>0$.
Further suppose that $4s_{\DD}(C_3\oplus P_4\oplus C_3)s_{\DD}(C_5\oplus C_5)>s_{\DD}(C_3\oplus P_2\oplus C_5)^2$.
If $s_{\DD}(C_3\oplus C_9)>0$, then $\DD$ is of Class II.
Otherwise, $\DD$ is of Class I if and only if $s_{\DD}(C_{12})>0$,
i.e., if $s_{\DD}(C_{12})=s_{\DD}(C_3\oplus C_9)=0$, then $\DD$ is of Class III.
\end{lemma}

\begin{proof}
Fix a $12$-deck $\DD$ with the properties as supposed in the statement of the lemma.
Note that $s_{\DD}(C_{\ell})=0$ for $\ell=4,6,8,10$ and
$s_{\DD}(C_3\oplus C_{\ell})=0$ for $\ell=5,7$ by Corollary~\ref{cor:deck10-III-classify}.
We first show that if $s_{\DD}(C_3\oplus C_9)>0$, then $\DD$ is of Class II.
Apply Lemma~\ref{lm:core} with $k=11$, $m=1$, any $\delta\in (0,1)$ such that $s_{\DD}(C_{12})\delta\le 1/2$,
$\gamma_3=\gamma_5=\gamma_7=\gamma_9=\gamma_{11}=0$, $\sigma_1=0$, $\tau_{1,3}=1$, $\tau_{1,5}=\tau_{1,7}=\tau_{1,11}=0$ and $\tau_{1,9}=-1$,
to get a non-zero kernel $U$ with the properties listed in Lemma~\ref{lm:core}.
Every principal graph $H$ with at most twelve edges satisfies that $s_{\DD}(H)=0$ or $t(H,U)=0$
unless $H$ is $C_{12}$ or $C_3\oplus C_9$.
Since $c^U_{\DD,2}=\cdots=c^U_{\DD,10}=0$ and $c^U_{\DD,12}\le -1/2$, the $12$-deck $\DD$ is of Class II.

We now prove that if $s_{\DD}(C_3\oplus C_9)=0$, then $\DD$ is of Class I or Class III.
Following the line of arguments presented in the proofs of Lemmas~\ref{lm:deck8-I} and~\ref{lm:deck10-I},
it is enough to establish that the deck $\DD$ is of Class III when $s_{\DD}(C_{12})=0$;
note that if $s_{\DD}(C_{12})=0$, then the deck $\DD$ is not of Class I by Lemma~\ref{lm:core-III}.

Let $U$ be an arbitrary non-zero kernel.
The coefficient $c^U_{\DD,2}$ is positive unless $U$ is balanced.
So, we can assume that $U$ is balanced.
It follows that $c^U_{\DD,4}=0$ and $c^U_{\DD,6}=s_{\DD}(C_3\cup C_3)t(C_3\cup C_3,U)\ge 0$ and
the equality holds only if $t(C_3,U)=0$.
Hence, we assume that $t(C_3,U)=0$ in the rest of the proof.
It follows that $c^U_{\DD,8}=0$ and
\begin{align*}
c^U_{\DD,10} & = s_{\DD}(C_5\cup C_5)t(C_5,U)^2 \\
             & + \int_{[0,1]}
                \begin{pmatrix}t_U^{C_3\oplus P_2}(x) \\ t_U^{C_5}(x) \end{pmatrix}^T
                \begin{pmatrix}
                s_{\DD}(C_3\oplus P_4\oplus C_3) & \frac{s_{\DD}(C_3\oplus P_2\oplus C_5)}{2} \\
                \frac{s_{\DD}(C_3\oplus P_2\oplus C_5)}{2} & s_{\DD}(C_5\oplus C_5)
                \end{pmatrix}
                \begin{pmatrix}t_U^{C_3\oplus P_2}(x) \\ t_U^{C_5}(x) \end{pmatrix}\dd x.
\end{align*}
Since the matrix in the integral above has both eigenvalues positive (as $4s_{\DD}(C_3\oplus P_4\oplus C_3)s_{\DD}(C_5\oplus C_5)>s_{\DD}(C_3\oplus P_2\oplus C_5)^2$),
we conclude that $c^U_{\DD,10}\ge 0$ and
the equality holds only if $t^{C_3\oplus P_2}_U(x)=t_U^{C_5}(x)=0$ for almost every $x\in [0,1]$.
Hence, if $c^U_{\DD,10}=0$, then $s_{\DD}(H,U)=0$ or $t(H,U)=0$ for every $12$-edge principal graph $H$.
It follows that if if $c^U_{\DD,10}=0$, then $c^U_{\DD,12}=0$.
We conclude that the $12$-deck $\DD$ is of Class III.
\end{proof}

\begin{figure}
\begin{center}
\epsfxsize 14cm
\epsfbox{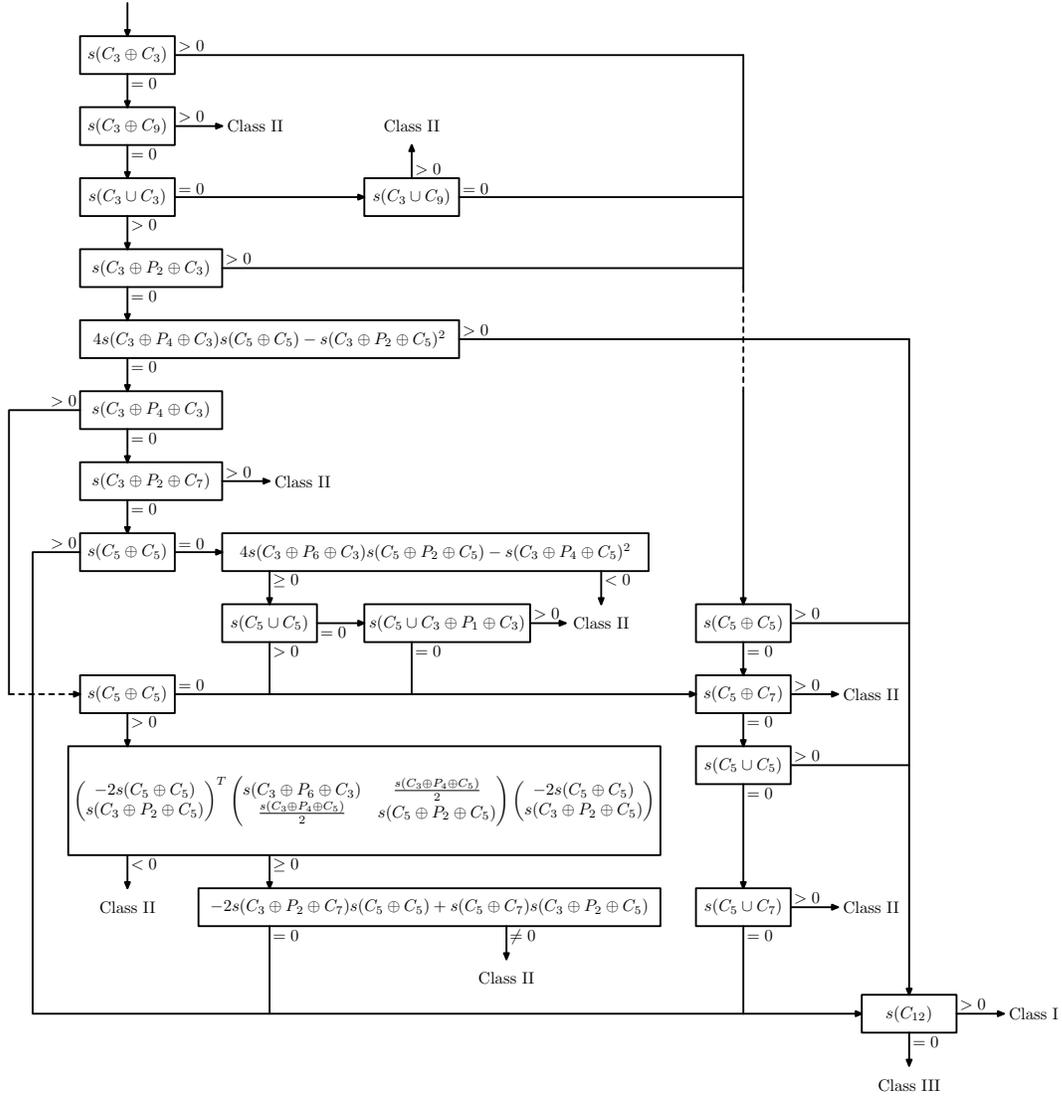}
\end{center}
\caption{The classification of $12$-decks with $s_{\DD}(P_2)>0$ whose $10$-deck is of Class III;
         we omit the subscript $\DD$ in the diagram.}
\label{fig:deck12}
\end{figure}

We are now ready to state the main theorem of this section.

\begin{theorem}
\label{thm:deck12}
Let $\DD$ be a $12$-deck with $s_{\DD}(P_2)>0$.
If the $10$-deck of $\DD$ is of Class I or of Class II, then $\DD$ is of Class I or of Class II, respectively.
Otherwise, the deck $\DD$ is of Class I, Class II or Class III as determined
in the diagram in Figure~\ref{fig:deck12}.
\end{theorem}

\begin{proof}
The proof follows by inspecting the diagram in Figure~\ref{fig:deck12} and verifying that
every path leading to the label Class I, Class II and Class III
is in line with the description given in Lemmas~\ref{lm:deck12-sC3C3}--\ref{lm:deck12-det}.
\end{proof}

\section{Conclusion} \label{sec:conclusion}

Our characterization of the possible behavior of the initial twelve terms of the polynomial
$t(H,1/2+\varepsilon U)+t(H,1/2-\varepsilon U)$
indicates that a complete characterization of locally common graphs would be complex.
Still, our results suggest the following two problems that could shed more light on a possible
characterization of locally common graphs.

\begin{problem}
\label{prob:deck}
Is it true that if two graphs $H$ and $H'$ have the same length $g$ of the shortest even cycle
and the same counts of principal graphs in their $g$-decks,
then both $H$ and $H'$ are either locally common or not locally common?
In other words, can it be determined whether $H$ is locally common based on the frequencies of principal graphs in the $g$-deck of $H$,
where $g$ is the length of the shortest even cycle in $H$?
\end{problem}

Problem~\ref{prob:deck} would follow from establishing that every deck containing an even cycle is of Class I or II;
we believe this to be the case but we were not able to find a short argument.
However, we do not have a suspected answer to offer to the next problem,
although Theorems~\ref{thm:deck8}, \ref{thm:deck10} and \ref{thm:deck12} suggests that the answer could also be positive.
In order to state the problem, we need to introduce a definition:
a graph $G$ is \emph{basic}
if every component of $G$ is a cycle or
isomorphic to $C_{\ell}\oplus P_n\oplus C_{\ell'}$ for a non-negative integer $n$ and some odd numbers $\ell$ and $\ell'$ (the value of
$\ell$, $\ell'$ and $n$ can be different for different components of $G$).

\begin{problem}
\label{prob:basic}
Is it true that if two graphs $H$ and $H'$ have the same length $g$ of the shortest even cycle
and the same counts of basic graphs in their $g$-decks,
then both $H$ and $H'$ are either locally common or not locally common?
In other words, can it be determined whether $H$ is locally common based on the frequencies of basic graphs in the $g$-deck of $H$,
where $g$ is the length of the shortest even cycle in $H$?
\end{problem}

\section*{Acknowledgement}

The authors would like to thank Martin Kure\v cka for his comments on the topics covered in this paper.

\bibliographystyle{bibstyle}
\bibliography{dcommon}

\end{document}